\newif\ifarxiv
\arxivtrue
\newif\ifpublish
\publishfalse
\ifpublish
\input{./prelude-asl.tex}
\else
\documentclass[a4paper]{article}
\usepackage[utf8]{inputenc} \usepackage[usenames,dvipsnames]{xcolor}
\usepackage[leqno]{amsmath} \usepackage{amssymb}
\usepackage{euscript} 
\usepackage{mathrsfs}
\usepackage{bbm}
\usepackage[amsmath,thmmarks]{ntheorem} 
\usepackage[ntheorem]{empheq}
\usepackage[normalem]{ulem} \usepackage[colorlinks=true,linkcolor=blue,anchorcolor=blue,citecolor=blue,
     		filecolor=blue,urlcolor=blue]{hyperref}
\usepackage{stmaryrd}\SetSymbolFont{stmry}{bold}{U}{stmry}{m}{n}\usepackage{bussproofs}\EnableBpAbbreviations
\usepackage{array}
\usepackage{tabularx}
\usepackage{enumitem}
\usepackage[small]{titlesec}
\usepackage[hmarginratio={1:1},top=3.1cm,bottom=3.85cm,textwidth=390pt]{geometry}
\usepackage[shortcuts]{extdash}
\usepackage{booktabs}
\usepackage[nameinlink,noabbrev,capitalize]{cleveref}

\usepackage{tocloft}
\usepackage{xspace}

\DeclareSymbolFont{sansops}{OT1}{\sfdefault}{m}{n}
\SetSymbolFont{sansops}{bold}{OT1}{\sfdefault}{b}{n}
\makeatletter
\renewcommand\operator@font{\mathgroup\symsansops}
\makeatother

\DeclareSymbolFont{bbold}{U}{bbold}{m}{n} \DeclareSymbolFontAlphabet{\mathbbold}{bbold}
\DeclareMathAlphabet{\mathpzc}{OT1}{pzc}{m}{it}

\numberwithin{equation}{section}
\setlist[enumerate]{label=(\arabic*)}

\newcommand{\dito}{\rightarrowtriangle}
\newcommand{\sco}{\,,\,}

\newcommand{\bbtwo}{\mathbbold{2}}

\newcommand{\Spaces}{\EuScript{S}}

\newcommand{\LFP}{\mathsf{LFP}}

\newcommand{\angs}[1]{\langle #1\rangle}

\newcommand{\pit}{\pi_2}

\newcommand{\mcc}{\mathcal{C}}

\newcommand{\mcj}{\mathcal{J}}
\newcommand{\mcl}{\mathcal{L}}
{}\newcommand{\mcs}{\mathcal{S}}
\newcommand{\mct}{\mathcal{T}}

\newcommand{\mfx}{\mathfrak{X}}

\newcommand{\FP}{{\mathsf{FP}}}
\newcommand{\FL}{{\mathsf{FL}}}

\newcommand{\catc}{\mathbb{C}}
\newcommand{\catj}{\mathbb{J}}

\newcommand{\comp}{\mathsf{p}}

{}\newcommand{\msep}{\mathrel|}		\newcommand{\op}{^\mathsf{op}}				\newcommand{\id}{\mathrm{id}}				

\newcommand{\ap}{\mathclose\cdot}			\newcommand{\qdot}{\,.\,}				

\newcommand{\setof}[2]{\{#1\msep #2\}}
\newcommand{\fami}[2]{(#1)_{#2}}

\newcommand{\N}{\mathbb{N}}

\newcommand{\too}{\longrightarrow}
\newcommand{\stoo}{\;\longrightarrow\;}
\newcommand{\simplies}{\;\Longrightarrow\;}
\newcommand{\mono}{\rightarrowtail}
\newcommand{\incl}{\hookrightarrow}
\newcommand{\epi}{\twoheadrightarrow}

\newcommand{\abs}[1]{\lvert#1\rvert}

\newcommand{\adj}{\dashv}
\newcommand{\ent}{\vdash}
\newcommand{\sent}{\;\vdash\;}

\def\signed #1{{\leavevmode\unskip\nobreak\hfil\penalty50\hskip2em
  \hbox{}\nobreak\hfil(#1)\parfillskip=0pt \finalhyphendemerits=0 \endgraf}}

\newsavebox\mybox

\newcommand{\gcons}{{\textstyle\int}}

\newcommand{\mor}{\mathsf{mor}}

\newcommand{\lrfami}[2]{\left(#1\right)_{#2}}

\newcommand{\image}{\mathsf{im}}
\newcommand{\subs}{\subseteq}

\newcommand{\classl}{\mathcal{L}}
\newcommand{\classr}{\mathcal{R}}

\newcommand{\fa}{\forall}
\newcommand{\ex}{\exists}

\makeatletter
\DeclareRobustCommand\onedot{\futurelet\@let@token\@onedot}
\def\@onedot{\ifx\@let@token.\else.\null\fi\xspace}
\newcommand{\ie}{i.e\onedot}

\newcommand{\aka}{a.k.a\onedot}
\newcommand{\wrt}{w.r.t\onedot}
\newcommand{\llp}{l.l.p\onedot}
\newcommand{\rlp}{r.l.p\onedot}

\newcommand{\ax}[1]{\AXC{$#1$}}
\newcommand{\ui}[1]{\UIC{$#1$}}
\newcommand{\bi}[1]{\BIC{$#1$}}
\newcommand{\ti}[1]{\TIC{$#1$}}

\newcommand{\drap}{\DP}

\let\lim\relax\DeclareMathOperator*{\lim}{\mathsf{lim}}
\DeclareMathOperator*{\colim}{\mathsf{colim}}
\newcommand{\bracetext}[1]{\{\text{#1}\}}
\newcommand{\braces}[1]{\{#1\}}

\DeclareFontFamily{U}{min}{}
\DeclareFontShape{U}{min}{m}{n}{<-> udmj30}{}
\newcommand\yo{\!\text{\usefont{U}{min}{m}{n}\symbol{'210}}}

\newcommand{\wfs}{w.f.s\onedot}
\newcommand{\locfp}{locally finitely presentable\xspace}
\newcommand{\lfp}{l.f.p\onedot}

\newcommand{\qtext}[1]{\quad\text{#1}\quad}

\newcommand{\pnc}{\pgfmatrixnextcell}
\newcommand{\oo}[1]{$\infty$\nobreakdash-\hspace{0pt}}

\newenvironment{inumerate}{\begin{enumerate}[label=\normalfont(\roman*)]}{\end{enumerate}}

\newcommand{\classe}{\mathsf{E}}
\newcommand{\classf}{\mathsf{F}}
\newcommand{\wfsef}{(\classe,\classf)}
{}\newcommand{\h}{\text{-}}
\newcommand{\clans}{\mcs}
\newcommand{\clant}{\mct}
{}\newcommand{\Alg}{\mathsf{Alg}}
\newcommand{\Set}{\mathsf{Set}}
\newcommand{\Cat}{\mathsf{Cat}}
\renewcommand{\AA}{\mathbb{A}}
\newcommand{\CC}{\mathbb{C}}
\newcommand{\DD}{\mathbb{D}}
\newcommand{\II}{\mathbb{I}}
\newcommand{\JJ}{\mathbb{J}}

\newcommand{\PP}{\mathbb{P}}
\newcommand{\RR}{\mathbb{R}}

\newcommand{\TT}{\mathbb{T}}

\makeatletter
\newcommand*{\twoheadrightsquigarrow}{\rightsquigarrow\joinrel\mathrel{\mathpalette\@twoheadrightsquigarrow\relax}}
\newcommand*{\@twoheadrightsquigarrow}[2]{\clipbox{{.7\width} 0pt 0pt {-.2\height}}{$\m@th#1\rightsquigarrow$}}

\newcommand{\congto}{\stackrel{\cong}{\to}}
\newcommand{\longsimto}{\stackrel{\simeq}{\longrightarrow}}
\newcommand{\Pos}{\mathsf{Pos}}
\newcommand{\idl}[2]{{#1}_{\leq #2}}
\newcommand{\sidl}[2]{{#1}_{< #2}}
\newcommand{\setmin}{{\setminus}}
\newcommand{\Lex}{\mathsf{Lex}}
\newcommand{\DLat}{\mathsf{DLat}}

\DeclareFontFamily{U} {MnSymbolC}{}
\DeclareFontShape{U}{MnSymbolC}{m}{n}{
  <-6> MnSymbolC5
  <6-7> MnSymbolC6
  <7-8> MnSymbolC7
  <8-9> MnSymbolC8
  <9-10> MnSymbolC9
  <10-12> MnSymbolC10
  <12-> MnSymbolC12}{}
\DeclareFontShape{U}{MnSymbolC}{b}{n}{
  <-6> MnSymbolC-Bold5
  <6-7> MnSymbolC-Bold6
  <7-8> MnSymbolC-Bold7
  <8-9> MnSymbolC-Bold8
  <9-10> MnSymbolC-Bold9
  <10-12> MnSymbolC-Bold10
  <12-> MnSymbolC-Bold12}{}
\DeclareSymbolFont{MnSyC} {U} {MnSymbolC}{m}{n}
\DeclareMathSymbol{\smalltriangleright}{\mathbin}{MnSyC}{72}
\newcommand{\codito}{\mathrel{\smalltriangleright}\mathrel{\!}\rightarrow}

\newcommand{\Ty}{\mathsf{Ty}}

\newcommand{\x}{\times}
\newcommand{\Mon}{\mathsf{Mon}}

\newcommand{\Fam}{\mathsf{Fam}}
\newcommand{\diag}[1]{[#1,\Set]}
\newcommand{\psh}[1]{\widehat{#1}}
\newcommand{\ini}{{i\in I}}
\newcommand{\z}[1]{$0$\nobreakdash-\hspace{0pt}}
\renewcommand{\^}[1]{\widehat{#1}}
\newcommand{\Mod}{\mathsf{Mod}}
\newcommand{\ox}{\otimes}
\newcommand{\opsimeq}{\stackrel{\mathsf{op}}{\simeq}}
\newcommand{\coclanc}{\mathcal{C}}
\newcommand{\cocland}{\mathcal{D}}

\newcommand{\czext}{compact $0$-ex\-ten\-sion\xspace}
\newcommand{\czexts}{compact $0$-ex\-ten\-sions\xspace}
\newcommand{\zext}{$0$-ex\-ten\-sion\xspace}
\newcommand{\zexts}{$0$-ex\-ten\-sions\xspace}
\newcommand{\fullmap}{\epi}
\newcommand{\extension}{\codito}
\newcommand{\cac}{clan-algebraic category\xspace}
\newcommand{\Clan}{\mathsf{Clan}}
\newcommand{\Kan}{\mathsf{Kan}}
\newcommand{\ClanAlg}{\mathsf{ClanAlg}}
\newcommand{\efcat}{$\wfsef$-category\xspace}
\newcommand{\efcats}{$\wfsef$-categories\xspace}
\newcommand{\EFCat}{\mathsf{EFCat}}
\newcommand{\czer}[1]{\mathfrak C(#1)}
\newcommand{\catcocont}{\mathsf{CoCont}}
\newcommand{\sm}{}
\newcommand{\cc}{_{\mathsf{cc}}}
\renewcommand{\dag}{_\dagger}
\newcommand{\CoClan}{\mathsf{CoClan}}
\newcommand{\efl}{\EuScript L}
\newcommand{\efm}{\EuScript M}

\newcommand{\clop}{\czer{\efl}\op}
\newcommand{\clopalg}{\clop\h\Mod}
\newcommand{\FC}{\mathsf{FC}}

\newcommand{\copalg}{\coclanc\op\!\h\Alg}
\newcommand{\fcc}{finite $\coclanc$-complex\xspace}
\newcommand{\fccs}{finite $\coclanc$-complexes\xspace}
\newcommand{\FCC}{\FC(\coclanc)}
\newcommand{\PD}{{(P,D)}}
\newcommand{\QE}{{(Q,E)}}
\newcommand{\hmod}{\h\Mod}

\newcommand{\PCC}{\widehat{\CC}}
\newcommand{\Colim}{\mathsf{Colim}}
\renewcommand{\comp}{\mathsf{comp}}
\newcommand{\CoMod}{\mathsf{CoMod}}
\renewcommand{\Lex}{\FL}
\newcommand{\ecxt}{\ast} \newcommand{\smod}{\clans\h\Mod}
\newcommand{\tmod}{\clant\!\h\Mod}
\newcommand{\dep}[1]{(#1)}

\newcommand{\El}{\mathsf{El}} 

\newcommand{\nerve}{_N}
\newcommand{\reali}{_\otimes}

\newcommand{\codisplay}{codisplay\xspace}
\newcommand{\sOA}{_{\mathsf{OA}}}
\newcommand{\sO}{_{\mathsf{O}}}
\newcommand{\sA}{_{\mathsf{A}}}

\newtheorem{theorem}{Theorem}[section]
  \AtBeginEnvironment{theorem}{\setlist[enumerate]{label=\normalfont(\roman*)}}
\newtheorem{proposition}[theorem]{Proposition}
  \AtBeginEnvironment{proposition}
        {\setlist[enumerate]{label=\normalfont(\roman*)}}
\newtheorem{lemma}[theorem]{Lemma}
  \AtBeginEnvironment{lemma}{\setlist[enumerate]{label=\normalfont(\roman*)}}

  \AtBeginEnvironment{conjecture}
        {\setlist[enumerate]{label=\normalfont(\roman*)}}
\newtheorem{corollary}[theorem]{Corollary}
  \AtBeginEnvironment{corollary}
        {\setlist[enumerate]{label=\normalfont(\roman*)}}

  \AtBeginEnvironment{claim}{\setlist[enumerate]{label=\normalfont(\roman*)}}
\theorembodyfont{\normalfont}
\theoremsymbol{\ensuremath{\diamondsuit}}

  \AtBeginEnvironment{exercise}{\setlist[enumerate]{label=(\alph*)}}
\newtheorem{remark}[theorem]{Remark}
  \AtBeginEnvironment{remark}{\setlist[enumerate]{label=(\alph*)}}

  \AtBeginEnvironment{remnot}{\setlist[enumerate]{label=(\alph*)}}

  \AtBeginEnvironment{notation}{\setlist[enumerate]{label=(\alph*)}}

  \AtBeginEnvironment{terminology}
        {\setlist[enumerate]{label=\normalfont(\alph*)}}
\newtheorem{remarks}[theorem]{Remarks}
  \AtBeginEnvironment{remarks}{\setlist[enumerate]{label=(\alph*)}}
\newtheorem{example}[theorem]{Example}
  \AtBeginEnvironment{example}{\setlist[enumerate]{label=(\alph*)}}
\newtheorem{examples}[theorem]{Examples}
  \AtBeginEnvironment{examples}{\setlist[enumerate]{label=(\alph*)}}

  \AtBeginEnvironment{thought}{\setlist[enumerate]{label=(\alph*)}}

  \AtBeginEnvironment{convention}{\setlist[enumerate]{label=(\alph*)}}

  \AtBeginEnvironment{conventions}{\setlist[enumerate]{label=(\alph*)}}
\newtheorem{void}[theorem]{$\!\!$}
  \AtBeginEnvironment{void}{\setlist[enumerate]{label=(\alph*)}}

\newtheorem{definition}[theorem]{Definition}
  \AtBeginEnvironment{definition}{\setlist[enumerate]{label=(\roman*)}}
\theoremsymbol{\ensuremath{_\square}}
\qedsymbol{\ensuremath{_\square}}
\theoremstyle{nonumberplain}
\theoremheaderfont{\itshape}
\newtheorem{proof}{Proof.}

\theorembodyfont{\itshape}
\theoremheaderfont{\normalfont\bfseries}

\makeatletter
\def\namedlabel#1#2{\begingroup
    #2\def\@currentlabel{#2}\phantomsection\label{#1}\endgroup
}
\makeatother

\ifarxiv
\usepackage{tikz-cd}
\newenvironment{mytikzcd}{\begin{tikzcd}}{\end{tikzcd}}
\newcommand{\tikzsetnextfilename}[1]{}
\else
\usepackage{tikz,etoolbox,environ}
\usetikzlibrary{cd,external}
\tikzsetexternalprefix{diagrams/}
\def\temp{&} \catcode`&=\active \let&=\temp
\newcommand{\mytikzcdcontext}[2]{
  \begin{tikzpicture}[baseline=(maintikzcdnode.base)]
    \node (maintikzcdnode) [inner sep=0, outer sep=0] {\begin{tikzcd}[#2]
        #1
    \end{tikzcd}};
  \end{tikzpicture}}
\NewEnviron{mytikzcd}[1][]{\def\myargs{#1}\edef\mydiagram{\noexpand\mytikzcdcontext{\expandonce\BODY}{\expandonce\myargs}}\mydiagram }
\fi

\newcommand{\Subsection}{Subsection~}
\newcommand{\Section}{Section~}
 \fi

\title{Duality for clans: an extension of Gabriel--Ulmer Duality}
\author{Jonas Frey}
\ifpublish
\revauthor{Frey, Jonas}
\address{Université Sorbonne Paris Nord, CNRS, Laboratoire d'Informatique de
Paris Nord, LIPN, F-93430 Villetaneuse, France}
\fi

\begin{document} 

\maketitle

\begin{abstract}
\emph{Clans} are categorical representations of generalized algebraic theories
that contain more information than the finite-limit categories associated to the
locally finitely presentable categories of models via Gabriel--Ulmer duality.
Extending Gabriel--Ulmer duality to account for this additional information, we
present a duality theory between clans and locally finitely presentable
categories equipped with a weak factorization system of a certain kind.
\end{abstract}

\setcounter{tocdepth}{1}\tableofcontents

\section{Introduction}\label{se:intro}

\emph{Gabriel--Ulmer duality}~\cite{gabriel1971lokal} is a contravariant
biequivalence 
\begin{equation*}
\Lex\;\opsimeq\;{\LFP}
\end{equation*}
between the $2$-category $\Lex$ of \emph{small finite-limit categories}, and the
$2$-category $\LFP$ of \emph{locally finitely presentable categories}, i.e.\
locally small cocomplete categories admitting a dense set of compact (a.k.a.\
finitely presentable) objects. The duality assigns to every small finite-limit
category $\mcc$ the category $\Lex(\mcc,\Set)$ of finite-limit preserving
functors to $\Set$, and conversely it associates to every locally finitely
presentable $\mfx$ the category $\LFP(\mfx,\Set)$ of finitary right adjoints
to $\Set$, which is equivalent to the opposite of the full subcategory
$\comp(\mfx)\subs\mfx$ of compact objects\footnote{Strictly speaking we have to
choose a small category which is \emph{equivalent} to $\comp(\mfx)\op$, since
the latter is only \emph{essentially small} in general.}.

We view Gabriel--Ulmer duality as a \emph{theory-model duality}: small
finite-limit categories $\mcc$ are viewed as theories (which we call
`finite-limit theories'), and---in the spirit of Lawverian functorial
semantics~\cite{lawvere1963functorial}---the functor category $\Lex(\mcc,\Set)$
is viewed as the category of \emph{models} of the finite-limit theory $\mcc$. 

It is well known that finite-limit theories are equally expressive as various
syntactically defined classes of theories, including
\begin{enumerate}
\item Freyd's \emph{essentially algebraic theories}~\cite{freyd1972aspects}, which permit a controlled form of partiality, \item Cartmell's \emph{generalized algebraic theories
(GATs)}~\cite{cartmell1978generalised,cartmell1986generalised}, which extend algebra by `dependent sorts', \item Johnstone's \emph{cartesian
theories}~\cite[Definition~D1.3.4]{elephant2}, which permit a limited form of existential quantification, 
and
\item\label{it:last} Palmgren--Vickers' \emph{Partial Horn
theories}~\cite{palmgren2007partial}, which are based on a calculus of first
order logic with partial terms and also admit relation symbols, 
\end{enumerate}
in the sense that for any theory $\TT$ from one of these classes, the category
$\TT\h\Mod$ of models is locally finitely presentable, and conversely for every
locally finitely presentable category $\mfx$ there exists a theory from that
class whose category of models is equivalent to $\mfx$. While from a certain
perspective this means that the classes (1)--\ref{it:last} of theories are all
equivalent to finite-limit theories, the syntactic representations of theories
contain additional information that is not reflected in the categories of
models, nor in the finite-limit theories. This `abstracting away' of syntactic
details is typically viewed as a \emph{strength} of the categorical/functorial
approach, and indeed in mathematical practice we no more want to distinguish
between the classical axiomatization of groups and Higman--Neumann's
axiomatization in terms of one operation and one
equation~\cite{higman1952groups}, than we want to distinguish between the
symmetric group $S_3$ and the dihedral group $D_3$.

However, it turns out that in the case of GATs, the theories contain additional
information that is not reflected in the corresponding finite-limit category,
but nevertheless goes beyond mere syntactic details. This information is related
to the structure of \emph{sort dependency} in the theories, and we show here
that it is reflected by certain weak factorization systems on the l.f.p.\
categories of models. For example, the $2$-sorted theory of graphs
\begin{align*}
            & \sent V
\\          & \sent E 
\\ x : E    & \sent s(x):V
\\ x : E    & \sent t(x):V
\end{align*}
with sorts $V$ and $E$ and source and target operations $s$ and $t$, has the
same category of models as the dependently sorted theory 
\begin{align*} & \sent V
\\ x\, y : V & \sent E(x,y)
\end{align*}
with a base sort $V$ of vertices and a dependent sort $E$ of edges (and
\emph{no} operations). But the syntactic categories of the theories are
different, and this is reflected by different weak factorization systems on the
categories of models: in the first case, the w.f.s.\ is cofibrantly generated by
the initial inclusions $\varnothing\incl (\bullet)$ and
$\varnothing\incl(\bullet{\to}\bullet)$ of the free graphs on one vertex and one
edge, respectively, whereas in the second case the w.f.s.\ is generated by the
inclusions $\varnothing\incl(\bullet)$ and
$(\bullet{\phantom{{\to}}}\bullet)\incl(\bullet{\to}\bullet)$. The non-trivial
domain of the second generator reflects the dependency of the sort of edges on
the sort of vertices.

\medskip

Concretely, the present work gives in Theorem~\ref{thm:main} a duality 
\begin{equation}\label{eq:bieq-clanalg}
\Clan\cc\;\opsimeq\;\ClanAlg
\end{equation}
of $2$-categories which extends Gabriel--Ulmer duality by incorporating this
additional structure. On the right we have the $2$-category of
\emph{clan-algebraic categories}, which are locally finitely presentable
categories equipped with a well-behaved kind of weak factorization system
(Definition \ref{def:cac}), while on the left we have a $2$-category of
Cauchy-complete \emph{clans} (Definition \ref{def:clan}). These are categorical
representations of GATs which can be viewed as a non-strict variant of
Cartmell's \emph{contextual categories} (Definition \ref{def:cxtcat}), and are
given by small categories equipped with a class of `display maps' representing
type families, admitting certain (but not all) finite limits. 

Besides extending Gabriel--Ulmer duality, the duality~\eqref{eq:bieq-clanalg}
recovers Adámek--Rosický--Vitale's duality between \emph{algebraic theories} and
\emph{algebraic categories}~\cite[Theorem 9.15]{adamek2010algebraic} as a
special case, and the latter duality was in fact inspirational for the present
work. See Remark~\ref{rems:char-algcat}\ref{rems:char-algcat-arv}.

\subsection{Structure of the paper}

\Section\ref{se:clans} introduces clans (Definition~\ref{def:clan}), the category of models of a clan $\tmod$ (Definition~\ref{def:tmod}), and the extension--full weak factorization system on models
(Definition~\ref{def:wfsef}). 

\Section\ref{se:comod-ump} gives a characterization of $\tmod$ as a kind of
cocompletion of $\clant\op$ (Theorem~\ref{thm:univ-prop-tmod}), and uses this to give presentations of slice categories $\tmod/A$, and of
coslice categories $H(\Gamma)/\tmod$ under representable models, as categories
of models of derived clans (Propositions \ref{prop:repr-coslice-tmod} and
\ref{prop:slice-tmod}).

\Section\ref{se:efcats} introduces in Definition~\ref{def:efcat} the auxiliary
notion of \emph{\efcat} (a l.f.p.\ category with a w.f.s.\ $\wfsef$), and shows
that the mapping $\clant\mapsto\tmod$ gives rise to a contravariant $2$-functor
from clans to \efcats which admits a left biadjoint
(Proposition~\ref{prop:biadj}).

\Section\ref{se:cc} shows that this biadjunction is idempotent, and that its
fixed points in clans are precisely the \emph{Cauchy complete} clans
(Definition~\ref{def:cauchy}). For this we use the notion of \emph{flat model}
(Definition~\ref{def:flat-model}), and the \emph{fat small object argument}, a
Corollary of which we state in Corollary~\ref{cor:fat-small}, but whose
systematic treatment we defer to Appendix~\ref{se:fsoa}.
Lemma~\ref{lem:compact-arrows} is an argument about compact objects in coslice
categories which was not found in the literature.

\Section\ref{se:cac} characterizes the fixed points of the biadjunction among
\efcats as \emph{clan-algebraic categories}, which are \efcats satisfying a
density and an exactness condition (Definition~\ref{def:cac}). The
characterization is given by Theorems~\ref{thm:tmod-clanalg} and
\ref{thm:clanalg-equiv}, where the proof of the latter requires a quite lot of
machinery including a Reedy-like resolution argument. As a consequence, we
obtain our main result in Theorem~\ref{thm:main}. As an application,
\Subsection\ref{suse:caws-cat} gives additional clan-algebraic w.f.s.s on the
example $\Cat$, which by the duality result correspond to additional
clan-representations of $\Cat$. 

\Section\ref{se:counterexample} contains a common counterexample to two natural
questions about clan-algebraic w.f.s.s, and \Section\ref{se:higher-types} 
discusses $\infty$-models of clans in higher types.

Appendix~\ref{se:lfp-soa} contains basic facts about locally finitely
presentable categories, weak factorization systems, and Quillen's small-object
argument, and Appendix \ref{se:gats} is an informal introduction to Cartmell's
generalized algebraic theories.

Finally, Appendix~\ref{se:fsoa} contains a careful development of the \emph{fat
small object argument} for clans. The fat small object is a variant of Quillen's
small object argument due to Makkai, Rosický, and Vokrinek~\cite{makkai2014fat}
(based on ideas by Lurie), which allows a more fine grained analysis of the
process of saturating a class of maps. We use it to show that \emph{\zexts are
flat} (Corollary~\ref{cor:zext-flat}), which is needed in Sections~\ref{se:cc}
and \ref{se:cac}.

\subsection{Acknowledgements}

Thanks to Steve Awodey, Andrew Swan, Fernando Larrain, and especially to Mathieu
Anel for many discussions on the topic of this paper. Thanks to Reid Barton for
telling me about the fat small object argument. Thanks to Benjamin Steinberg for
locating the reference~\cite{head1982/83expanded} for me after I asked about it
on MathOverflow\footnote{\url{https://mathoverflow.net/a/90747/51432}}. Thanks
to Kian Cho for feedback on the manuscript and catching many typos.

This material is based upon work supported by the Air Force Office of Scientific
Research under award number FA9550-20-1-0305, and the Army Research Office under
award number W911NF-21-1-0121.

\section{Clans}\label{se:clans}

\begin{definition}\label{def:clan} A \emph{clan} is a small category $\clant$
with a distinguished class $\clant\dag$
of arrows called \emph{display maps}, such that:
\begin{enumerate}
\item Pullbacks of display maps along arbitrary maps exist and are display maps,
i.e.\  if $p:\Gamma'\dito\Gamma$ is a display map and $s:\Delta\to\Gamma$ is
arbitrary, then there exists a pullback square
\vspace{-1.5ex}
\begin{equation}\label{eq:display-pullback}
\tikzsetnextfilename{p2luowhb}
\begin{mytikzcd}
    |[label={[overlay, label distance=-2.5mm]-45:\lrcorner}]|
    \Delta'  \ar[r, "s'"]
                \ar[-{Triangle[open]},d, "q"']
\pnc   \Gamma'  \ar[-{Triangle[open]},d, "p"]
\\    \Delta    \ar[r, "s"]
\pnc   \Gamma
\end{mytikzcd}
\end{equation}
where $q$ is a display map.
\item Isomorphisms and compositions of display maps are display maps.
\item $\clant$ has a terminal object, and terminal projections are display maps.
\end{enumerate}
A \emph{clan morphism} is a functor between clans which preserves display maps,
pullbacks of display maps, and the terminal object. We write $\Clan$ for the
$2$-category of clans, clan-morphisms, and natural transformations.
\end{definition}
\begin{remarks}
\begin{enumerate}
\item Definition \ref{def:clan} (apart from the smallness condition), and the term
`display map', were introduced by Taylor in his
thesis~\cite[4.3.2]{taylor1987recursive}, the explicit link to Cartmell's work
was made in his textbook~\cite[Chapter~VIII]{taylor1999practical}. The name
`clan' is due to Joyal~\cite[Definition~1.1.1]{joyal2017notes}.
\item Following Cartmell, we use the arrow symbol $\dito$ for display maps.
\item We have defined clans to be \emph{small} by default, since this fits with
our point of view of clans as theories, and makes the duality theory work.

However, it is also reasonable to consider non-small, `semantic' clans, and we
will mention them occasionally (e.g.\ in
Example~\ref{ex:clan-examples}\ref{ex:clan-examples:kan} below), using the term
\emph{large clan} in this case.
\end{enumerate}
\end{remarks} 

\begin{examples}\label{ex:clan-examples}
\begin{enumerate}
\item\label{ex:clan-examples:fl-clan}
Small finite-limit categories can be viewed as clans where \emph{all}
morphisms are display maps. We call such clans \emph{finite-limit clans}.
\item\label{ex:clan-examples:fp-clan} Small finite-product categories can be
viewed as clans where the display maps are the morphisms that are (isomorphic
to) product projections. We call such clans \emph{finite-product clans}.
\item\label{ex:clan-examples:kan} $\Kan$ is the \emph{large} clan whose
underlying category is the full subcategory of the category $[\Delta\op,\Set]$
of simplicial sets on \emph{Kan complexes}, and whose display maps are the
\emph{Kan fibrations}.
\item\label{ex:clan-examples:gat} The syntactic category of every
\emph{generalized algebraic theory} in the sense of
Cartmell~\cite{cartmell1978generalised,cartmell1986generalised} is a clan. This
is explained in \Section\ref{se:gats}, and we discuss the example of the clan
for categories in greater detail in \Subsection \ref{suse:cat-clan} below.
\end{enumerate}
\end{examples}
Since it seems to lead to a more readable exposition, we introduce explicit
notation and terminology for the dual notion.
\begin{definition}
A \emph{coclan} is a small category $\coclanc$ with a distinguished class
$\coclanc\dag$ of arrows called \emph{\codisplay maps} satisfying the dual axioms
of clans. The $2$-category $\CoClan$ of coclans is defined dually to that of
clans, i.e.\
\begin{equation*}
\CoClan(\coclanc,\cocland)=\Clan(\coclanc\op,\cocland\op)\op
\end{equation*}
for coclans $\coclanc,\cocland$. We use the arrow symbol $\codito$ for
co-display maps.
\end{definition}
\begin{remark}\label{rem:henry-cofib-cat} Coclans appear under the name
\emph{cofibration categories} in \cite[Def~2.1.5]{henry2016algebraic}. This is
however in conflict with Baues' notion of cofibration category, which also
includes a notion of weak equivalence~\cite[Section~I.1]{baues1989algebraic}.
See also Remark~\ref{rem:wfsef}\ref{rem:wfsef:henry} below.
\end{remark}

\subsection{Models}\label{suse:models}

\begin{definition}\label{def:tmod}A \emph{model} of a clan $\clant$ is a functor
$A:\clant\to\Set$ which preserves the terminal object and pullbacks of display
maps. We write $\tmod$ for the category of models of $\clant$, viewed as a full
subcategory of the functor category $\diag\clant$.
\end{definition}
\begin{remark}
In other words, a model of a clan $\clant$ is a clan morphism into the large
clan with underlying category $\Set$ set and the maximal (i.e.\ finite-limit)
clan structure.

In the spirit of functorial semantics, it is possible to consider models of
clans in other categories than sets, and even in other (typically large) clans.
However, the duality theory presented here is about models in $\Set$ and we
don't consider any other kind (apart from some speculations about
$\infty$-categorical models in \Section \ref{se:higher-types}).
\end{remark}

\begin{examples}
\begin{enumerate}
\item If $\mcc$ is a finite-limit clan
(Example \ref{ex:clan-examples}\ref{ex:clan-examples:fl-clan}) then $\Mod(\mcc)$
coincides with the category $\FL(\mcc,\Set)$ of finite-limit preserving functors
into $\Set$, which we also view as category of models of $\mcc$ qua finite-limit
theory. This means that it makes sense to view finite-limit theories as a
special case of clans.
\item If $\mcc$ is a finite-product clan, then $\Mod(\mcc)$ is the category
$\FP(\mcc,\Set)$ of finite-product preserving functors into $\Set$.

In Adámek, Rosický and Vitale's textbook~\cite[Def.~1.1]{adamek2010algebraic},
small finite-product categories are called \emph{algebraic theories}, and models
of algebraic theories are defined to be finite-product preserving functors into
$\Set$. Thus, we recover their notions as a special case, i.e.\ finite-product
clans correspond to algebraic theories, and models correspond to algebras. To
emphasize the analogy to the finite-limit case, we refer to algebraic theories
also as \emph{finite-product theories}.
\item If $\TT$ is a GAT, then the category of models of its syntactic category
$\mcc[\TT]$ (with the clan structure described in \Subsection
\ref{suse:gat-syncat}) is equivalent to the models of $\TT$, which Cartmell
defines\footnote{Strictly speaking, Cartmell does not `define' the models of
$\TT\h\Mod$ to be $\mathsf{ConFunc}(\mcc[\TT],\Fam)$ but `asserts' that the
categories are equivalent~\cite[pg.~2.77]{cartmell1978generalised}. But since he
refrains from giving a formal definition of $\TT\h\Mod$---writing only `It
should be quite clear what we mean by
model'~\cite[pg.~1.45]{cartmell1978generalised}---we take the assertion as a
definition.} to be the category $\mathsf{ConFunc}(\mcc[\TT],\Fam)$ of contextual
functors and natural transformations into the contextual category $\Fam$ of
iterated families of sets.
\end{enumerate} 
\end{examples}

The following remarks discuss some categorical properties of the category
$\tmod$ of models of a clan, establishing in particular that it is locally
finitely presentable. We refer to \Section \ref{se:lfp-soa} for the relevant
definitions.
\begin{remarks}\label{rem:mods}
\begin{enumerate}
\item\label{rem:mods:refl-lim-fcolim}As a category of models of a finite-limit sketch, $\tmod$ is reflective (and
therefore closed under arbitrary limits) in $\diag\clant$, and moreover it is
closed under filtered colimits~\cite[Section 1.C]{adamek1994locally}. In
particular, $\tmod$ is \locfp. 
\item\label{rem:mods:clan-to-mods}The representable functors $\clant(\Gamma,-):\clant\to\Set$ are models of
$\clant$ for all $\Gamma\in\clant$, thus the Yoneda embedding
$\yo:\clant\op\to\diag\clant$ lifts along the inclusion $\tmod\incl\diag\clant$
to a fully faithful functor $H:\clant\op\to\tmod$.
\tikzsetnextfilename{cpYXlqrK}
\begin{equation*}
\begin{mytikzcd}
\pnc   {\tmod}
        \ar[d,hook]
\\  {\clant\op}
        \ar[ru,"H",dashed]
        \ar[r,"\yo"'] \pnc   {\diag\clant}
\end{mytikzcd}
\end{equation*}
\vspace{-4mm}

\item
For $\Gamma\in\clant$, the representable functor
\begin{equation*}
    \tmod(H(\Gamma),-):\tmod\to\Set
\end{equation*}
is isomorphic to the evaluation functor $A\mapsto A(\Gamma)$, hence it preserves
filtered colimits as those are computed in $\diag\clant$ and therefore
pointwise. This means that $H(\Gamma)$ is \emph{compact}\footnote{Following
Lurie~\cite{lurie2009higher} we use the shorter term `compact' instead of the
more traditional `finitely presented' for objects whose covariant representable
functor preserves filtered colimits.} in $\tmod$.
\end{enumerate}
\end{remarks}

\subsection{The clan for categories}\label{suse:cat-clan}

\Subsection \ref{suse:gat-syncat} describes how the syntactic category $\mcc[\TT]$ of every
GAT $\TT$ can be viewed as a clan. The present section elaborates this for the
specific case of the GAT $\TT_\Cat$ of categories~\eqref{eq:gat-cat}. We will
use this clan and variations as a running example throughout the article.

Recall from Definition \ref{def:syncat} that the objects of $\mct_\Cat:=\mcc[\TT_\Cat]$
are equivalence classes of contexts, and the arrows are equivalence classes of
substitutions. By inspection of the axioms we see that sorts in $\TT_\Cat$
cannot depend on non-variable terms, since the only non-constant sort symbol is
$x\,y:O\ent A\dep{x,y}$ and there are no function symbols of type~$O$. This
means that up to reordering, all contexts are of the form
\begin{equation}\label{eq:cat-cxt}
(x_1\dots x_n:O,y_1:A\dep{x_{s_1},x_{t_1}},\dots,y_k:A\dep{x_{s_k},x_{t_k}})
\end{equation}
where $n,k\geq 0$ (such that $n> 0$ whenever $k>0$) and $1\leq s_l,t_l\leq n$
for $1\leq l\leq k$; declaring first a list of object variables and then a list
of arrow variables, each depending on a pair of the object variables. Given
another context $(u_1\dots u_m,v_1 \dots v_h)$, a substitution 
\begin{equation*}
u_1\dots u_m,v_1 \dots v_h\sent \sigma \;:\; x_1\dots x_n,y_1 \dots y_k
\end{equation*}
is a tuple $\sigma = (u_{i_1}\dots u_{i_n},f_1\dots f_k)$ where 
$1\leq i_1,\dots,i_n\leq m$ and the $f_l$ are terms
\begin{equation*}
u_1\dots u_m,v_1 \dots v_h\sent f_l\;:\; A(u_{i_{s_l}},u_{i_{t_l}}).
\end{equation*}
Some reflection shows that $\mcc[\TT_\Cat]$ is dual to the full subcategory of
$\Cat$ on free categories on finite graphs: the data of a
context~\eqref{eq:cat-cxt} is that of finite sets $V=\braces{x_1,\dots,x_n}$,
$E=\braces{y_1,\dots,y_k}$ of \emph{vertices} and \emph{edges}, and source
and target functions $s,t:E\to V$, and a substitution $\sigma$ as above consists
of a mapping from $\braces{x_1,\dots,x_n}$ to $\braces{u_1,\dots,u_m}$ and a
mapping from $\braces{y_1,\dots,y_k}$ to suitable paths in the graph represented
by the domain. This is not surprising, since every clan embeds contravariantly
into its category of models by Remark~\ref{rem:mods}\ref{rem:mods:clan-to-mods}.
Finally, the display maps in $\mct_\Cat$, which syntactically correspond to
projections `from longer contexts to shorter ones', correspond to functors
$G^*\incl H^*$ between free categories induced by inclusions (i.e.\
monomorphisms) $G\incl H$ of finite graphs in the dual presentation.

\subsection{The weak factorization system on models}

Next we introduce the \emph{extension--full weak factorization system} on the
category of models of a clan. We refer to \Section \ref{se:lfp-soa} for basic
facts about lifting properties and weak factorization systems (w.f.s.s) as well
as pointers to the literature.
\begin{definition}\label{def:wfsef}Let $\clant$ be a clan.
\begin{enumerate}
\item We call a map $f:A\to B$ in $\tmod$ \emph{full}, if it has the \emph{right
lifting property} (\rlp, see~Definition \ref{def:wfs}\ref{def:wfs:llp-rlp}) \wrt
all maps $H(p)$ for $p$ a display map.
\item We call $f:A\to B$ an \emph{extension}, if it has the \emph{left lifting
property} (\llp) \wrt all full maps. 
\item We call $A\in\tmod$ a \emph{\zext}, if $0\to A$ is an extension.
\end{enumerate}
\end{definition}
\begin{remarks}\label{rem:wfsef}
\begin{enumerate}
\item\label{rem:wfsef:arrow-symbols}We use the arrow symbols `$\codito$' for extensions (just as for \codisplay
maps), and `$\fullmap$' for full maps. We write $\classe$ and $\classf$ for the
classes of extensions and full maps in $\tmod$, respectively. By the \emph{small
object argument} (Theorem~\ref{thm:soa}), extensions and full maps form a
weak factorization system $\wfsef$.
\item\label{rem:wfsef:wpb-regepi}A map $f:A\to B$ in $\tmod$ is full if and only if the naturality square
\begin{equation*}
\tikzsetnextfilename{srwFEAZF}
\begin{mytikzcd}
    A(\Delta)
        \ar[r,"A(p)"]
        \ar[d,"f_\Gamma"']
\pnc   A(\Gamma)
        \ar[d,"f_\Delta"]
\\  B(\Delta)
        \ar[r,"B(p)"] 
\pnc   B(\Gamma)
\end{mytikzcd}
\end{equation*}
is a \emph{weak pullback}\footnote{Meaning that the comparison map to the actual
pullback is a surjection.} in $\Set$ for all display maps $p:\Delta\dito\Gamma$.
Setting $\Gamma=1$ we see that full maps are pointwise surjective and therefore
regular epimorphisms (the pointwise kernel pair $p,q: R\to A$ of $f$ is in
$\tmod$ since $\tmod\incl\diag\clant$ creates limits, and pointwise surjective
maps are coequalizers of their kernel pairs in $\diag\clant$, hence all the more
so in $\tmod$).
\item\label{rem:wfsef:display-extension}For every display map $p:\Delta\dito\Gamma$ in $\clant$, the arrow
$H(p):H(\Gamma)\extension H(\Delta)$ is an extension---these are precisely the
generators of the \wfs. In particular, all representable models $H(\Gamma)$ are
\zexts, since all terminal projections $\Gamma\dito 1$ are display maps in
$\clant$.
\item\label{rem:wfsef:henry} The same w.f.s.\ was already defined by Simon Henry
in~\cite[Definition~2.4.2]{henry2016algebraic}, using the terminology of
`cofibration categories' mentioned in Remark~\ref{rem:henry-cofib-cat}. There,
extensions are called cofibrations, and full maps trivial fibrations. We have
not used this homotopical terminology here since we don't want to think about
full maps as being `trivial' in any way.
\end{enumerate}
\end{remarks}
\begin{examples}\label{ex:wfsef}
\begin{enumerate}
\item\label{ex:wfsef:fp}If $\clant$ is a finite-product clan, then $\wfsef$ is cofibrantly generated by
initial injections $0\cong H(1)\extension H(\Gamma)$, since for every display
map $p:\Delta\times\Gamma\dito\Delta$ the generator $H(p)$ is a pushout
\begin{equation*}
\tikzsetnextfilename{NMqSaNfz}
\begin{mytikzcd}
    H(1)
        \ar[r,{Triangle[open,reversed]->}]
        \ar[d,{Triangle[open,reversed]->}]
\pnc   H(\Delta)
        \ar[d,"H(p)",{Triangle[open,reversed]->}]
\\  H(\Gamma)
        \ar[r,{Triangle[open,reversed]->}]
\pnc   |[label={[overlay, label distance=-3mm]135:\ulcorner}]|
    H(\Gamma\times\Delta)
\end{mytikzcd}
\end{equation*}
in $\tmod$ of an initial inclusion, and left classes of \wfs{}s are closed under
pushout. It follows that the full maps are \emph{precisely} the pointwise
surjective maps, which in this case also coincide with the regular epis, since
finite-product preserving functors are closed under image factorization in
$\diag\clant$ (and thus every non-surjective arrow factors through a strict
subobject). Thus, the \zexts are precisely the regular-projective objects in the
finite-product case, which also play a central role
in~\cite{adamek2010algebraic}.
\item\label{ex:wfsef:fl}If $\clant$ is a finite-limit clan then \emph{all} naturality squares of
full maps $f:A\fullmap B$ are weak pullbacks, including the naturality squares
\begin{equation*}
\tikzsetnextfilename{NiTWluiY}
\begin{mytikzcd}
    A(\Gamma)
        \ar[r]
        \ar[d]
\pnc   A(\Gamma\x\Gamma)\cong A(\Gamma)\x A(\Gamma)
        \ar[d]
\\  B(\Gamma)
        \ar[r]
\pnc   B(\Gamma\x\Gamma)\cong B(\Gamma)\x B(\Gamma)
\end{mytikzcd}
\end{equation*}
of diagonals $\Gamma\to\Gamma\x\Gamma$. From this it follows easily that
$f_\Gamma$ is injective, and since we have shown that it is surjective above, we
conclude that only isomorphisms are full in the finite-limit case.
\item\label{ex:wfsef:cat} The w.f.s.\ on $\Cat$ induced by the presentation
$\Cat=\mct_\Cat\hmod$ (see \Subsection\ref{suse:cat-clan}) has as full maps
functors $F:\CC\to\DD$ which have the r.l.p.\ w.r.t.\ all functors $G^*\incl
H^*$ for inclusions $G\incl H$ of finite graphs. It is not difficult to see that
these are precisely the functors which are full in the classical sense and
moreover surjective on objects, and that the w.f.s.\ is already generated by the
functors $(0\incl 1)$ and $(2\incl\bbtwo)$, where $2$ is the discrete category
with two objects and $\bbtwo$ is the interval category.
\end{enumerate}
\end{examples}

\section{Comodels and the universal property of \texorpdfstring{$\tmod$}{T-Alg}}
\label{se:comod-ump}

\subsection{Nerve--realization adjunctions}

We recall basic facts about \emph{nerve--realization} adjunctions, to establish
notation and conventions. Recall that for small $\CC$ the presheaf category
$\PCC=[\CC\op,\Set]$ is the \emph{small-colimit completion} of $\CC$, in the
sense that for every cocomplete category $\mfx$, precomposition with the Yoneda
embedding $\yo:\CC\to\PCC$ induces an equivalence 
\begin{equation}\label{eq:pcc-ump}
\catcocont(\PCC,\mfx)\longsimto[\CC,\mfx]
\end{equation}
between  the categories of  cocontinuous functors $\PCC\to\mfx$, and of functors
$\CC\to\mfx$. Specifically, the cocontinuous functor $F\reali:\PCC\to\mfx$
corresponding to $F:\CC\to\mfx$ is the {left Kan extension} of $F$ along
$\yo:\CC\to\PCC$, whose value at $A\in\^\CC$ admits alternative representations
\begin{align*}
F\reali(A)=F\ox A&=\textstyle\int^{C\in\CC}F(C)\x A(C)
\\&=\colim(\El(A)\to\CC\xrightarrow{F}\mfx)
\end{align*}
as a coend and as a colimit indexed by the category $\El(A)$ of elements of $A$.
If $\mfx$ is locally small then $F\reali$ has a right adjoint
$F\nerve:\mfx\to\PCC$ given by $F\nerve(X)=\mfx(F(-),X)$. We call $F\nerve$ and
$F\reali$ the \emph{nerve} and \emph{realization} functors of $F$, respectively,
and $F\reali\adj F\nerve$ the \emph{nerve--realization adjunction} of $F$.

\subsection{Comodels and the universal property of \texorpdfstring{$\tmod$}{T-Mod}}

The universal property of $\tmod$ is an equivalence between cocontinuous
functors out of $\tmod$ and coclan morphisms out of $\clant\op$. Following a
suggestion by Mathieu Anel, we refer to the latter as \emph{comodels} of the
clan. We will only use this term for coclan morphisms with cocomplete codomain.
\begin{definition}
A \emph{comodel} of a clan $\clant$ in a cocomplete category $\mfx$ is a functor
$F:\clant\op\to\mfx$ which sends $1$ to $0$, and display-pullbacks to pushouts.
We write $\clant\!\h\CoMod(\mfx)$ for the category of comodels of $\clant$ in
$\mfx$, as a full subcategory of the functor category.
\end{definition}
\begin{remark}
In other words, a comodel of $\clant$ in $\mfx$ is a coclan morphism from 
$\clant\op$ to the large coclan with underlying category $\mfx$ and the maximal 
coclan structure.
\end{remark}
\begin{theorem}[The universal property of $\tmod$]\label{thm:univ-prop-tmod} 
Let
$\clant$ be a clan. 
\begin{enumerate}
\item\label{thm:univ-prop-tmod:h-comod}The functor $H:\clant\op\to\tmod$ from
Remark~\ref{rem:mods}\ref{rem:mods:clan-to-mods} is a comodel.
\item\label{thm:univ-prop-tmod:equiv}For every cocomplete $\mfx$ and comodel $F:\clant\op\to\mfx$, the restriction of
$F\reali:[\clant,\Set]\to\mfx$ to $\tmod$ is cocontinuous. Thus, precomposition
with $H$ gives rise to an equivalence 
\begin{equation}\label{eq:univ-prop-talg}
\catcocont(\tmod,\mfx)\longsimto\clant\!\h\CoMod(\mfx)
\end{equation}
between categories of continuous functors and of comodels.
\item\label{thm:univ-prop-tmod:nerve}If $F:\clant\op\to\mfx$ is a comodel and $\mfx$ is locally small, then the nerve
functor $F\nerve:\mfx\to[\clant,\Set]$ factors through the inclusion
$\tmod\incl[\clant,\Set]$, giving rise to a \emph{restricted nerve realization
adjunction} $F\reali:\tmod\leftrightarrows\mfx:F\nerve$.
\tikzsetnextfilename{BfhewVAX}
\begin{equation*}\begin{mytikzcd}[row sep = large]
    {\clant\op}       
        \ar[dr,"F"', shorten > = 5pt]
        \ar[r,"H",hook]
\pnc   {\tmod}
        \ar[d,"F\reali"' near start, shift right = 2.5, shorten > = 2pt]
        \ar[d,phantom,"\adj"]
        \ar[from=d,"F\nerve"' near end, shift right = 2.5, shorten < = 2pt]
        \ar[r,hook,gray]
\pnc   |[gray]|{[\clant,\Set]}
        \ar[dl,"F\reali"' pos=.35, shift right = 1.5,gray, shorten = 8pt]
        \ar[ld,phantom,"\adj",gray, shift left = 1]
        \ar[from=dl,"F\nerve"' pos=.6, shift right = 3.5,gray, shorten = 8pt]
\\\pnc \mfx
\end{mytikzcd}\end{equation*}
\end{enumerate}
\end{theorem}
\begin{proof}
Analogous statements to \ref{thm:univ-prop-tmod:h-comod} and
\ref{thm:univ-prop-tmod:equiv} hold more generally for arbitrary small
realized\footnote{A sketch is called `realized' if all its designated cones are
limiting.} limit sketches. As Brandenburg points out on
MathOverflow\footnote{https://mathoverflow.net/q/403653}, the earliest reference
for this seems to be~\cite[Theorem~2.5]{pultr1970right}. See
also~\cite{brandenburg2021large} which gives a careful account of an even more
general statement for non-small sketches.

For claim \ref{thm:univ-prop-tmod:nerve}, it's easy to see that for $X\in\mfx$,
the functor $F\nerve(X)=\mfx(F(-),X)$ is a model since $F$ is a comodel.
\end{proof}

\subsection{Slicing and coslicing}

As an application of Theorem~\ref{thm:univ-prop-tmod}, this subsection gives
statements about clan presentations of slice categories $\tmod/A$ of categories
of models (Proposition~\ref{prop:slice-tmod}), and of coslice categories
$H(A)/\tmod$ under representable models
(Proposition~\ref{prop:repr-coslice-tmod}).

\begin{definition}
For $\clant$ a clan and $\Gamma\in\clant$, we write $\clant_\Gamma$ for the
full subcategory of $\clant\!/\Gamma$ on display maps. Then $\clant_\Gamma$ is a
clan where an arrow in $\clant_\Gamma$ is a display map if its underlying map
is one in $\clant$. Compare~\cite[Proposition~1.1.6]{joyal2017notes}.
\end{definition}
\begin{proposition}\label{prop:repr-coslice-tmod}Let $\Gamma$ be an object of a clan $\clant$.
Then the functor 
\begin{equation}\label{eq:coslice-functor}
\Gamma/H:(\clant_\Gamma)\op\to H(\Gamma)/\tmod
\end{equation} 
which sends $d:\Delta\dito\Gamma$ to $H(d):H(\Gamma)\codito H(\Delta)$ is a
comodel. Moreover, its restricted  nerve--realization adjunction (in the sense
of Theorem~\ref{thm:univ-prop-tmod}\ref{thm:univ-prop-tmod:nerve})
\begin{equation}\label{eq:Gamma/H-nra}
\tikzsetnextfilename{BfhewVAY}
\begin{mytikzcd}[row sep = large]
        {(\clant_\Gamma)\op}
            \ar[dr,"\Gamma/H"', shorten > = 5pt]
            \ar[r,"H",hook]
\pnc    {\clant_\Gamma\hmod}
            \ar[d,"(\Gamma/H)\reali"' near start, shift right = 2.5, shorten > = 2pt]
            \ar[d,phantom,"\adj"]
            \ar[from=d,"(\Gamma/H)\nerve"' near end, shift right = 2.5, shorten < = 2pt]
\\\pnc {H(\Gamma)/\tmod}
\end{mytikzcd}
\end{equation}
is an equivalence and identifies the extension--full \wfs on
$\clant_\Gamma\hmod$ with the coslice \wfs on $H(\Gamma)/\tmod$.
\end{proposition}
\begin{proof}
It is easy to see that $\Gamma/H$ is a comodel. For the second claim, since
arrows $H(\Gamma)\to A$ correspond to elements of $A(\Gamma)$, we can identify
the coslice category $H(\Gamma)/\tmod$ with the category of `$\Gamma$-pointed
models of $\clant$', \ie pairs $(A,x)$ of a model $A$ and an element
$x\in A(\Gamma)$, and morphisms preserving chosen elements. 

Under this identification, we first verify that the functor $(\Gamma/H)\nerve$ is
given by 
\begin{equation*}
(\Gamma/H)\nerve(A,x)(\Delta\stackrel{d}{\dito}\Gamma)=\setof{y\in A(\Delta)}{d\cdot 
    y = x}, 
\end{equation*}
and then that it is an equivalence with inverse $\Phi:\clant_\Gamma\hmod\to
H(\Gamma)/\tmod$ given by 
\begin{equation*}
\Phi(B)=(B(-\times\Gamma\stackrel{\pit}{\dito}\Gamma),\delta\cdot \star)
\end{equation*}
where $\star$ is the unique element of $B(\id_\Gamma)$ and
$\delta:\Gamma\to\Gamma\times\Gamma$ is the diagonal map viewed as global
element of $\pit:\Gamma\x\Gamma\dito\Gamma$ in $\clant_\Gamma$. 
Thus, $(\Gamma/H)\reali=\Phi$.

Finally we note that the w.f.s.\ on $H(\Gamma)/\tmod$
is cofibrantly generated by commutative triangles
\begin{equation}\label{eq:gen-proj}
\tikzsetnextfilename{qFUGebtd}
\begin{mytikzcd}[column sep = large]
        H(\Gamma)       
            \ar[dr,"H(\pit)",{Triangle[open,reversed]}->]
            \ar[d,"H(\pit)"',{Triangle[open,reversed]}->]
\pnc\\  H(\Delta\x\Gamma)       
            \ar[r,"H(d\x\Gamma)"',{Triangle[open,reversed]}->]
\pnc    H(\Theta\x\Gamma)
\end{mytikzcd}\end{equation}
for display maps
$d:\Theta\dito\Gamma$~\cite[Theorem~2.7]{hirschhorn2021overcategories}. On the
other hand, since $(\Gamma/H)\reali\circ H = \Gamma/H$ (see~\eqref{eq:Gamma/H-nra}),
the functor $(\Gamma/H)\reali$ sends the generators of the extension--full w.f.s.\ on
$\clant_\Gamma\hmod$ to triangles
\begin{equation}\label{eq:gen-gen}
\tikzsetnextfilename{qOyoyNnT}
\begin{mytikzcd}
        H(\Gamma)       
            \ar[dr,"H(f)",{Triangle[open,reversed]}->]
            \ar[d,"H(e)"',{Triangle[open,reversed]}->]
\pnc\\  H(\Delta)       
            \ar[r,"H(d)"',{Triangle[open,reversed]}->]
\pnc    H(\Theta)
\end{mytikzcd}\end{equation}
for arbitrary display maps $d,e,f$ in $\clant$. Now the triangles of
shape~\eqref{eq:gen-gen} contain the triangles  of shape \eqref{eq:gen-proj},
but are contained in their saturation, which is the left class of the coslice
w.f.s.\ Thus, the two w.f.s.s are equal.
\end{proof}

\begin{proposition}\label{prop:slice-tmod}Let $A$ be a model of a clan $\clant$. Then the projection functor
$\El(A)\to\clant\op$ creates a coclan structure on $\El(A)$, \ie $\El(A)$ is a
coclan with \codisplay maps those arrows that are mapped to display maps in
$\clant$. Moreover, the canonical functor
\begin{equation*}
H/A:\El(A)\cong\clant\op/A\to\tmod/A
\end{equation*}
is a comodel, and its restricted nerve--realization adjunction 
\begin{equation*}
\tikzsetnextfilename{BfhewVAZ}
\begin{mytikzcd}[row sep = large]
    {\El(A)}
        \ar[dr,"H/A"', shorten > = 5pt]
        \ar[r,"H",hook]
\pnc   {\El(A)\op\hmod}
        \ar[d,"(H/A)\reali"' near start, shift right = 2.5, shorten > = 2pt]
        \ar[d,phantom,"\adj"]
        \ar[from=d,"(H/A)\nerve"' near end, shift right = 2.5, shorten < = 2pt]
\\\pnc {\tmod/A}
\end{mytikzcd}
\end{equation*}
is an equivalence which identifies the extension--full w.f.s.\ on
$\El(A)\op\hmod$ and the slice w.f.s.\ on $\tmod/A$.
\end{proposition} 
\begin{proof}
The verification that $\El(A)\op$ is a clan and $H/A$ is a coclan morphism is
straightforward. The equivalence is a restriction of the well-known equivalence
$\widehat{\clant\op}/A\simeq\widehat{\clant\op/A}$. The w.f.s.s coincide
since---again by $(H/A)\reali\circ H=H/A$---the functor $(H/A)\reali$ sends the
generators of the w.f.s.\ on $\El(A)\op\hmod$ to commutative triangles 
\tikzsetnextfilename{PezDXLjb}
\begin{equation*}\begin{mytikzcd}
    H(\Gamma)
            \ar[dr,"\hat x"']
            \ar[r,"d", {Triangle[open,reversed]}->]
\pnc   H(\Delta)
            \ar[d,"\hat y"]
\\\pnc A
\end{mytikzcd}\end{equation*} 
in $\tmod/A$, where $d:\Delta\dito\Gamma$ is a display map in $\clant$ and $x\in
A(\Gamma)$ and $y\in A(\Delta)$ are elements with $d\cdot y = x$.
By~\cite[Theorem~1.5]{hirschhorn2021overcategories}, these form a set of
generators for the slice w.f.s.\ on $\tmod/A$.
\end{proof}

\section{\texorpdfstring{\efcats}{(E,F)-categories)} and the biadjunction}
\label{se:efcats}

\begin{definition}\label{def:efcat}An \efcat is a l.f.p.\ category $\efl$ with a \wfs $\wfsef$ whose maps we call
\emph{extensions} and \emph{full maps}. A \emph{morphism  of \efcats} is a
functor $F:\efl\to\efm$ preserving small limits, filtered colimits, and full
maps. We write $\EFCat$ for the $2$-category of \efcats, morphisms of \efcats,
and natural transformations.
\end{definition}
\begin{lemma}\label{lem:ef-ladj}
If $F:\efl\to\efm$ is a morphism of \efcats, then it
has a left adjoint $L:\efm\to\efl$ which preserves compact objects and
extensions. Conversely, if $L:\efm\to\efl$ is a cocontinuous functor preserving
compact objects and extensions, then it has a right adjoint $F:\efl\to\efm$
which is a morphism of \efcats. Writing $\EFCat_L(\efm,\efl)$ for the category
of cocontinuous functors $\efm\to\efl$ preserving extensions and compact
objects, we thus have $\EFCat_L(\efm,\efl)\simeq\EFCat(\efl,\efm)\op$.
\end{lemma}
\begin{proof}
That morphisms of \efcats have left adjoints follows from the \emph{adjoint
functor theorem for presentable
categories}~\cite[Theorem~1.66]{adamek1994locally}, and conversely the
\emph{special adjoint functor theorem}~\cite[Section~V-8]{maclanecwm} implies
that cocontinuous functors between \lfp categories have right adjoints. It
follows from standard arguments that the left adjoint preserves compact objects
iff the right adjoint preserves filtered colimits, and that the left adjoint
preserves extensions iff the right adjoint preserves full maps.
\end{proof} 
\begin{lemma}\label{lem:precomp-inverse-ef}
For any morphism $F:\clans\to\clant$ of clans, the precomposition
functor 
\begin{equation*}
(-)\circ F\;:\;\tmod\;\to\;\smod
\end{equation*}
is a morphism of \efcats. Thus the assignment $\clant\mapsto\tmod$ extends to a
\mbox{$2$-functor}
\begin{equation*}
(-)\hmod\;:\;\Clan\sm\op\;\to\;\EFCat
\end{equation*}
from clans to \efcats.
\end{lemma}
\begin{proof}
The preservation of small limits and filtered colimits is obvious since they are
computed pointwise (Remark~\ref{rem:mods}\ref{rem:mods:refl-lim-fcolim}). To
show that $(-\circ F)$ preserves full maps, let $f:A\to B$ be full in $\tmod$.
It is sufficient to show that the $(f\circ F)$-naturality squares are weak
pullbacks at all display maps $p:$ in $\smod$. But the $(f\circ F)$-naturality
square at $p$ is the same as the $f$-naturality square at $F(p)$ so the claim
follows since $f$ is full and $F$ preserves display maps.
\end{proof}
\begin{definition}\label{def:czer}
Given an \efcat $\efl$, write $\czer{\efl}\subs\efl$ for the full subcategory 
on \czexts. 
\end{definition}
\begin{lemma}
$\czer{\efl}$ is a coclan with extensions as \codisplay maps. \qed
\end{lemma}
\begin{proposition}\label{prop:biadj}The assignment $\efl\mapsto\czer{\efl}\op$ extends to a pseudofunctor
\begin{equation*}
\czer{-}\op\;:\;\EFCat\to\Clan\op
\end{equation*}
which is left biadjoint to $(-)\hmod:\Clan\op\to\EFCat$.
\end{proposition}
\begin{proof}
We show that for every \efcat $\efl$, the $2$-functor
\begin{equation*}
\EFCat(\efl,(-)\hmod):\Clan\sm\op\to\Cat
\end{equation*}
is birepresented by $\czer{\efl}\op$. Given a clan $\clant$ it is easy to see
that the equivalence
\begin{equation*}
\catcocont(\tmod,{\efl})\;\simeq\;\clant\!\h\CoMod(\efl)
\end{equation*}
from Theorem~\ref{thm:univ-prop-tmod} restricts to an equivalence 
\begin{equation*}
\EFCat_L(\tmod,\efl)\;\simeq\;\CoClan\sm(\clant\op,\czer{\efl}).
\end{equation*}
Taking opposite categories on both sides we get 
\begin{equation}\label{eq:nat-equiv}
\EFCat(\efl,\tmod)\;\simeq\;\Clan\sm(\clant,\czer{\efl}\op)
\end{equation}
as required.
\end{proof}
\begin{remark}
From the construction of the natural equivalence~\eqref{eq:nat-equiv} we can 
extract explicit descriptions of the components 
\begin{equation*}
\Theta_\efl:\efl\to\clopalg\qquad\text{and}\qquad 
E_\clant:\clant\to\czer{\tmod}\op \end{equation*}
of the {unit} $\Theta$ and the {counit} $E$ of the biadjunction
\begin{equation}\label{eq:biadj}
\czer{-}\op\;:\;\EFCat\;\leftrightarrows\;\Clan\sm\op\;:\;(-)\hmod
\end{equation}
at an \efcat $\efl$ and a clan $\clant$ respectively. Specifically,
$\Theta_\efl$ is the nerve of the inclusion $J:\czer{\efl}\incl\efl$ (which is
obviously a comodel), and $E_\clant$ is $(-)\op$ of the evident corestriction of
$H:\clant\op\to\tmod$.
\end{remark}
In \Section\ref{se:cc} and \Section\ref{se:cac} we show that the
biadjunction~\eqref{eq:biadj} is \emph{idempotent} (in the sense that the
associated monad and comonad are), and characterize the fixed points on both
sides (Theorems~\ref{thm:cc-equiv} and \ref{thm:clanalg-equiv}).

\section{Cauchy complete clans and the fat small object argument}
\label{se:cc}

\begin{definition}\label{def:cauchy}
A clan $\clant$ is called \emph{Cauchy complete} if its underlying category is
Cauchy complete (\ie idempotents split), and retracts of display maps are
display maps.
\end{definition}
\begin{examples}\label{ex:cc}
\begin{enumerate}
\item\label{ex:cc:fl}Finite-limit clans are always Cauchy complete, since finite-limit
categories are and all arrows are display maps in finite-limit clans.
\item\label{ex:cc:fp}A finite-product clan is Cauchy complete if and only if idempotents split in the
underlying finite-product category, which may or may not be the case for the
presentation of a single-sorted algebraic theory $\TT$ as \emph{Lawvere theory}
(i.e.\ the opposite of the full subcategory of $\Mod(\TT)$ on \emph{finitely
generated free models}). For example the Lawvere theory of abelian groups is
Cauchy complete since all finitely presented projective abelian groups are free,
whereas the Lawvere theory of distributive lattices is \emph{not} Cauchy
complete. A non-free retract of a finitely generated free distributive lattice
may be obtained by starting with a section--retraction pair $s:\braces{0 < 1 <
2}\leftrightarrows\braces{0<1}^2:r$ in {posets}, and then taking the
distributive lattice of \emph{upper sets} on both sides, i.e.\ applying the
functor $\Pos(-,\braces{0<1}):\Pos\op\to\DLat$. Then
$\Pos(\braces{0<1}^2,\braces{0 < 1})$ is the free distributive lattice on $2$
generators, but $\Pos(\braces{0<1<1},\braces{0<1})$ is not free.

Further details on the question of Cauchy-completeness of finite-limit theories,
including a discussion of how the classical theory of \emph{Morita equivalence
of rings} fits into the picture, can be found in \cite[Sections 8,
15]{adamek2010algebraic}.
\item\label{ex:cc:cat}The clan $\mct_\Cat$ of categories {is} Cauchy complete. To see this assume
that $G$ is a finite graph and that 
$\DD$ is a retract of the free category $G^*$ on $G$. Then we know that $\DD$ is
a compact $0$-extension and we have to show that $\DD$ is free on a finite
graph. Call an arrow $f$ in $\DD$ \emph{irreducible} if it is not an identity
and in any decomposition $f = gh$, either $g$ or $h$ is an identity. Since the
factors of every non-trivial decomposition have shorter length in $G^*$, every
arrow in $\DD$ admits a decomposition into irreducible factors. Let $H$ be the
graph of irreducible arrows in $\DD$, and let $F:H^*\to\DD$ be the canonical
functor. Then $F$ is full since all arrows in $D$ are composites of
irreducibles, and it admits a section $K:\DD\to H^*$ since $\DD$ is a
$0$-extension. As a section, $K$ sends arrows in $\DD$ to decompositions into
irreducibles, thus it sends irreducible arrows to themselves. It follows that
$K(F(j))=j$ for generators $j$ in $H$, and from this we can deduce that $K\circ
F=\id_{H^*}$. Thus, $\DD\cong H^*$. Finiteness of $H$ follows from compactness.

This argument is an adaption of a similar argument for
monoids~\cite{head1982/83expanded}.

\item\label{ex:cc:czer}For every \efcat $\efl$, the clan $\czer{\efl}\op$ (Definition \ref{def:czer}) is Cauchy
complete, since compact objects and extensions are closed under retracts.
\end{enumerate}
\end{examples}
By Example~\ref{ex:cc}\ref{ex:cc:czer}, Cauchy completeness is a necessary
condition for the counit $E_\clant:\clant\to\czer{\tmod}\op$ of the
biadjunction~\eqref{eq:biadj} to be an equivalence. We will show that it is also
sufficient, but for this we need the notion of \emph{flat model}, and the
\emph{fat small object argument}.

Recall that for small $\CC$, a functor $F:\CC\to\Set$ is called \emph{flat} if
$\El(F)$ is filtered, or equivalently if $F\reali:[\CC\op,\Set]\to\Set$ preserves
finite limits~\cite[Definition 6.3.1 and Proposition 6.3.8]{borceux1}. From the
second characterization it follows that flat functors preserve all finite limits
that exist in $\catc$, thus for the case of a clan $\clant$, flat functors
$F:\clant\to\Set$ are always models. We refer to them as \emph{flat} models:

\begin{definition}\label{def:flat-model} A model $A:\clant\to\Set$ of a clan
$\clant$ is called \emph{flat}, if $\El(F)$ is filtered.
\end{definition}
\begin{lemma}\label{lem:flat-filtered}A $\clant$-model $A$ is flat iff it is a filtered colimit of representable
models. 
\end{lemma}
\begin{proof}
We always have $A=\colim(\El(A)\to\clant\op\xrightarrow{H}\tmod)$, thus if $A$
is flat then it is a filtered colimit of representable models. The other
direction follows since representable models are flat, and flat functors are
closed under filtered colimits in
$\diag\clant$~\cite[Proposition~6.3.6]{borceux1}.
\end{proof}
\begin{corollary}\label{cor:fat-small}For any clan $\clant$, the \z-extensions in $\tmod$ are flat.
\end{corollary}
\begin{proof}
This follows from the \emph{fat small object argument} and can be seen as a
special case of~\cite[Corollary~5.1]{makkai2014fat}, but we give a direct proof
in Appendix~\ref{se:fsoa} (Corollary~\ref{cor:zext-flat}).
\end{proof}
\begin{definition}\label{def:compact-arrow}
Let $\mfx$ be a cocomplete locally small category.
\begin{enumerate}
\item We say that an arrow $f:A\to B$ is \emph{orthogonal} to a small
diagram $D:\JJ\to\mfx$, and write $f\perp D$, if the following square is a
pullback in $\Set$.
\begin{equation*}
\tikzsetnextfilename{HahOMTXT}
\begin{mytikzcd}
    \colim_{j\in\JJ}\mfx(B,D_j)
        \ar[r,""']
        \ar[d,""]
\pnc	\mfx(B,\colim(D))
        \ar[d,""]
\\	\colim_{j\in\JJ}\mfx(A,D_j)
        \ar[r,""]
\pnc	\mfx(A,\colim(D))
\end{mytikzcd}
\end{equation*}
\item We call $f$ \emph{compact} if it is orthogonal to all small filtered
diagrams.
\end{enumerate} 
\end{definition}
\begin{lemma}\label{lem:compact-arrows}
Let $\mfx$ be a locally small cocomplete category.
\begin{inumerate}
\item\label{lem:compact-arrows:obj-compact-init-incl} An object $A\in\mfx$ is
compact in the usual sense that $\mfx(A,-)$ preserves filtered colimits, if and
only if the arrow $0\to A$ is compact in the sense of
Definition~\ref{def:compact-arrow}. 
\item\label{lem:compact-arrows:compact-cancel-compo}
If the arrow $g$ in a commutative triangle
\tikzsetnextfilename{iNHuiIsf}
$\begin{mytikzcd}[sep = small]
    A       \ar[dr,"f"']
            \ar[r,"g"]
\pnc   B       \ar[d,"h"] \\\pnc C \end{mytikzcd}$ is compact, then $f$ is compact if
and only if $h$ is compact.  In other words, compact arrow are closed under
composition and have the right cancellation property.
\item\label{lem:compact-arrows:arrow-coslice-obj}
If $f:A\to B$ is compact as an arrow in $\mfx$, then it is compact as an
object in $A/\mfx$.
\item\label{lem:compact-arrows:arrow-between-compact}If $h:B\to C$ is an arrow between compact objects in $\mfx$, then $h$ is
compact as an object in $B/\mfx$.
\end{inumerate}
\end{lemma}
\begin{proof}
\ref{lem:compact-arrows:obj-compact-init-incl} is obvious, and 
\ref{lem:compact-arrows:compact-cancel-compo} follows from the pullback lemma.

For \ref{lem:compact-arrows:arrow-coslice-obj} assume that $f$ is compact as an
arrow in $\mfx$ and consider a filtered diagram in $A/\mfx$, given by a filtered
diagram $D:\II\to \mfx$ and a cocone $\gamma=(\gamma_i:A\to D_i)_{i\in\II}$.
Note that since the forgetful functor $A/\mfx\to\mfx$ creates connected
colimits, we have $\colim(\gamma):A\to\colim(D)$. Also because $\II$ is
connected, all $\gamma_i$ are in the same equivalence class in
$\colim_{i\in\II}\mfx(A,D_i)$, which we denote by
$\overline{\gamma}:1\to\colim_{i\in\II}\mfx(A,D_i)$. We have to show that the
canonical map
\begin{equation*}
\textstyle\colim_i(A/\mfx)(f,\gamma_i)\stoo(A/\mfx)(f,\colim(\gamma))
\end{equation*}
is a bijection. 
This follows because this function can be presented by a pullback in
$\Set^\bbtwo$ as in the following diagram.
\begin{equation*}
\tikzsetnextfilename{hfMHhIbt}
\begin{mytikzcd}[column sep = -25, row sep = 5]
\pnc   \colim_i\mfx(B,D_i)
            \ar[rr]
            \ar[dd]
\pnc\pnc  \mfx(B,\colim(D))
\ar[dd]
\\   |[gray]|\colim_i(A/\mfx)(f,\gamma_i)
            \ar[ru,gray]
            \ar[dd,gray]
            \ar[rr,crossing over,gray]
\pnc\pnc   |[gray]|(A/\mfx)(f,\colim(\gamma))
            \ar[ru,gray]
\\\pnc \colim_i\mfx(A,D_i)
            \ar[rr]
\pnc\pnc  \mfx(A,\colim(D))
\\  1       
            \ar[ru,"{\overline\gamma}"]
            \ar[rr]
\pnc\pnc  1
            \ar[ru,"\colim(\gamma)"']
            \ar[from=uu,crossing over,gray]
\end{mytikzcd}
\end{equation*}
The front square is a pullback since the back one is by compactness of $f$ as an
arrow, and since the side ones are pullbacks by construction. Thus the gray
horizontal arrow is a bijection since $1\to 1$ is.

Finally, claim \ref{lem:compact-arrows:arrow-between-compact} now follows
directly from \ref{lem:compact-arrows:obj-compact-init-incl},
\ref{lem:compact-arrows:compact-cancel-compo}, and
\ref{lem:compact-arrows:arrow-coslice-obj}.
\end{proof}
\begin{remark}
One can show the implication of
Lemma~\ref{lem:compact-arrows}\ref{lem:compact-arrows:arrow-coslice-obj} is
actually an equivalence, \ie $f:A\to B$ is compact as an arrow if and only if it
is so as an object of the coslice category, but the other direction is more
awkward to write down and we don't need it.
\end{remark}

\begin{theorem}\label{thm:cc-equiv}
If $\clant$ is a Cauchy complete clan, then $E_\clant:\clant\to\czer{\tmod}\op$
is an equivalence of clans.
\end{theorem}
\begin{proof}
Let $C\in\tmod$ be a \czext. Then by Corollary~\ref{cor:fat-small}, $C$ is a
filtered colimit of representable models, and since $C$ is compact the identity
$\id_C$ factors through one of the colimit inclusions, whence $C$ is a retract
of a representable model. By Cauchy completeness, $C$ is thus itself
representable, i.e.\ we have an equivalence of categories.

It remains to show that $E_\clant$ reflects extensions to display maps. Assume
$f:\Delta\to \Gamma$ in $\clant$ such that $H(f): H(\Gamma)\to H(\Delta)$ is an
extension. Then $H(f)$ is compact in $H(\Gamma)/\tmod$ by
Lemma~\ref{lem:compact-arrows}\ref{lem:compact-arrows:arrow-between-compact} and
$H(\Gamma)/\tmod\simeq \clant_\Gamma\hmod$ by
Proposition~\ref{prop:repr-coslice-tmod}. This means that the object
corresponding to $H(f)$ in $\clant_\Gamma\hmod$ is a compact \z-extension, and
thus it is isomorphic to a representable model $\clant_\Gamma(d,-)$ for a
display map $d:\Theta\dito\Gamma$ by the argument in the first part of the
proof. This means that $f$ is isomorphic to $d$ over $\Gamma$, and therefore a
display map.
\end{proof}
The preceding proposition together with Example~\ref{ex:cc}\ref{ex:cc:czer}
shows that the pseudomonad on $\Clan$ induced by the biadjunction
\eqref{eq:biadj} is \emph{idempotent}: applying the pseudomonad once produces a
Cauchy complete clan, and applying it again gives something equivalent. By
general facts about (bi)adjunctions, the induced pseudomonad on $\EFCat$ is also
idempotent. In the next section we characterize its fixed points as being
\emph{clan-algebraic categories}.

\section{Clan-algebraic categories}\label{se:cac}

\begin{definition}\label{def:cac}
An \efcat $\efl$ is called \emph{clan-algebraic} if 
\begin{description}
\item[\namedlabel{def:cac:d}{(D)}] the inclusion
$J:\czer{\efl}\hookrightarrow\efl$ is dense,
\item[\namedlabel{def:cac:cg}{(CG)}] the w.f.s.\ $\wfsef$ is cofibrantly
generated by $\classe\cap\czer{\efl}$, and
\item[\namedlabel{def:cac:fq}{(FQ)}] equivalence relations $\angs{p,q}:R\mono
A\x A$ in $\efl$ with \emph{full components} $p,q$ are effective, and have full
coequalizers.
\end{description}
A \emph{clan-algebraic weak factorization system} is a w.f.s.\ on a l.f.p.\
category $\efl$ making $\efl$ into a clan-algebraic category.
\end{definition}
\begin{theorem}\label{thm:tmod-clanalg}
The category $\tmod$ is clan-algebraic for every clan $\clant$.
\end{theorem}
\begin{proof}
Conditions \ref{def:cac:d} and \ref{def:cac:cg} are straightforward. For for
\ref{def:cac:fq} let $\langle p,q\rangle : R\mono A\times A$ be an equivalence
relation with full components. This means that we have an equivalence relation
$\sim$ on each $A(\Gamma)$, such that
\begin{itemize}
\item for all arrows $s : \Delta\to\Gamma$, the function $A(s) = s\cdot(-) :
A(\Delta)\to A(\Gamma)$ preserves this relation, and
\item for every display map $p : \Gamma^+\dito\Gamma$ and all $a,b\in A(\Gamma)$
and $c\in A(\Gamma^+)$ such that $a\sim b$ and $p\cdot c = a$, there exists a
$d\in A(\Gamma^+)$ with $c\sim d$ and $p\cdot d=b$.
\end{itemize}
We show first that the pointwise quotient $A/R$ is a model. Clearly
$(A/R)(1)=1$, and it remains to show that given a pullback
\begin{equation*}
\tikzsetnextfilename{sMBFcJCW}
\begin{mytikzcd}
    \Delta^+\ar[r,"t"]
            \ar[d,"g",-{Triangle[open]}] 
\pnc   \Gamma^+\ar[d,"f",-{Triangle[open]}] 
\\  \Delta  \ar[r,"s"] 
\pnc   \Gamma
\end{mytikzcd}
\end{equation*}
with $f$ and $g$ display maps, and elements $a\in A(\Delta)$, $b\in A(\Gamma^+)$
with $s\cdot a\sim f\cdot b$, there exists a unique-up-to-$\sim$ element $c\in
A(\Delta^+)$ with $g\cdot c\sim a$ and $t\cdot c\sim b$. Since $f$ is a display
map, there exists a $b'$ with $b\sim b'$ and $v\cdot b'=s\cdot a$, and since $A$
is a model there exists therefore a $c$ with $q\cdot c = a$ and $t\cdot c = b'$.
For uniqueness assume that $c,c'\in A(\Delta^+)$ with $q\cdot c\sim q\cdot c'$
and $t\cdot c\sim t\cdot c'$. Then $c\sim c'$ follows from the fact that $R$ is
a model. This shows that $A/R$ is a model, and also that the quotient is
effective, since the kernel pair is computed pointwise. The fact that $A\to A/R$
is full is similarly easy to see.
\end{proof}
The following counterexample shows that conditions \ref{def:cac:d} and
\ref{def:cac:cg} alone are not sufficient to characterize categories $\tmod$.
\begin{example}\label{def:non-calg}
\newcommand{\Inj}{\mathsf{Inj}} Let $\Inj$ be the full subcategory of $\^\bbtwo$
on injections, and let $\wfsef$ be the w.f.s.\ generated by $0\to\yo(0)$ and
$0\to\yo(1)$. Then $\wfsef$ satisfies \ref{def:cac:d} and \ref{def:cac:cg}, and
$\classf$ consists precisely of the pointwise surjective maps, in particular it
contains all split epis. However, the equivalence relation on $\id_2$ which is
discrete on the domain and codiscrete on the codomain is not effective.
\end{example}
The following is a restatement of
Remark~\ref{rem:wfsef}\ref{rem:wfsef:wpb-regepi} for clan-algebraic categories.
\begin{lemma}\label{lem:calg-full-repi}
Full maps in clan-algebraic categories are regular epimorphisms.
\end{lemma}
\begin{proof}
Given a full map in a clan-algebraic category $\efl$, the lifting property
against \z-extensions implies that $\Theta_\efl(f)=J\nerve(f)$ is componentwise
surjective in $\clopalg$, and therefore the coequalizer of its kernel pair.
Since left adjoints preserve regular epis, we deduce that $J\reali(J\nerve(f))$
is regular epic in $\efl$ and the claim follows since $J\reali\circ J\nerve\cong
\id$ by \ref{def:cac:d}.
\end{proof}
\begin{remark}
Observe that we only used property \ref{def:cac:d} in the proof, no exactness. 
\end{remark}
\begin{lemma}\label{lem:cac-full-rcancel}
The class $\classf$ of full maps in a \cac $\efl$ has the 
\emph{right cancellation property}, \ie we have $g\in\classf$ whenever 
$gf\in\classf$ and $f\in\classf$ for composable pairs $f:A\to B$, 
$g:B\to C$.
\end{lemma}
\begin{proof} By \ref{def:cac:cg} it suffices to show that $g$ has the
\rlp \wrt extensions $e:I\extension J$ between \czexts $I,J$. Let 
\begin{equation*}
\tikzsetnextfilename{DlKJVzKa}
\begin{mytikzcd}
    I       \ar[d,{Triangle[open,reversed]->},"e"]
            \ar[r,"h"]
\pnc   B       \ar[d,"g"]
\\  J       \ar[r,"k"]
\pnc   C
\end{mytikzcd}
\end{equation*}
be a filling problem.
Since $I$ is a \z-extension and $f$ is full, there exists a map $h':I\to A$
with $fh'=h$. We obtain a new filling problem
\begin{equation*}
\tikzsetnextfilename{IAUqkgmJ}
\begin{mytikzcd}
    I       \ar[d,{Triangle[open,reversed]->},"e"]
            \ar[r,"h'"]
\pnc   A       \ar[d,"gf",two heads]
\\  J       \ar[r,"k"]
\pnc   C
\end{mytikzcd}
\end{equation*}
which can be filled by a map $m:J\to A$ since $gf$ is full. Then $m'=fm$ is a 
filler for the original problem (the upper triangle commutes since $m'e=fme=fh'=h$).
\end{proof}
\begin{lemma}\label{lem:imfac} Let $\efl$ be a clan-algebraic category, let
$f:A\to B$ be an arrow in $\efl$ with componentwise full kernel pair $p,q:R\epi
A$, and let $e:A\epi C$ be the coequalizer of $p$ and $q$. Then the unique
$m:C\to B$ with $me=f$ is monic.
\end{lemma}
\begin{proof}
By~\ref{def:cac:d} it is sufficient to test monicity of $m$ on maps out of
compact \z-extensions $E$. Let $h,k:E\to C$ such that $mh = mk$. Since $e$ is
full by~\ref{def:cac:fq}, there exist $h',k':E\to A$ with $eh'=h$ and $ek'=k$.
In particular we have $fh'=fk'$ and therefore there is a unique $u:E\to R$ with
$pu=h'$ and $qu=k'$. Thus we can argue
\begin{equation*}
h = eh' = epu = equ = e k' = k
\end{equation*}
which shows that $m$ is monic.
\end{proof}

\begin{lemma}\label{lem:flat-junit-iso} If $A\in\clopalg$ is flat then $A\to
J\nerve(J\reali(A))$ is an isomorphism. Thus, $J\reali$ restricted to flat
models is fully faithful.
\end{lemma}
\begin{proof}
For the fist claim we have 
\begin{align*}
J\nerve(J\reali(A))(C) &= \efl(C,\colim(\El(A)\to \czer{\efl}\incl\efl))
\\&\cong \colim(\El(A)\to\czer{\efl}\xrightarrow{\yo(C)}\Set)
        &&\text{since $\El(A)$ is filtered}
\\&\cong \yo(C)\ox A\cong A(C).
\end{align*}
The second claim follows since for flat $B$, the mapping 
\begin{equation*}
(\clopalg)(A,B)\to \efl(J\reali (A),J\reali (B))
\end{equation*}
can be decomposed as
\begin{equation*}
(\clopalg)(A,B)\to(\clopalg)(A,J\nerve(J\reali(B)))
               \to\efl(J\reali (A),J\reali (B)).
\end{equation*}
\end{proof}

\begin{lemma}\label{lem:jfc}
The following are equivalent for a cone $\phi:\Delta C\to D$ on a diagram
$D:\catj\to\efl$ in an \efcat $\efl$.
\begin{enumerate}
\item Given an extension $e:A\to B$, an arrow $h:A\to C$, and a cone
$\kappa:\Delta B\to D$ such that $\phi_j h=\kappa_j e$ for all $j\in
\catj$, there exists $l:B\to C$ such that $le=h$ and $\phi_jl=\kappa_j$ for all
$j\in\catj$.
\tikzsetnextfilename{EbmsYGLD}
\begin{equation*}\begin{mytikzcd}
    A
        \ar[r,"h"]
        \ar[d,"e"']
\pnc	C
        \ar[d,"\phi_j"]
\\	B
        \ar[ru,"l", dashed]
        \ar[r,"\kappa_j"]
\pnc	D_j
\end{mytikzcd}\end{equation*}
\item The mediating arrow $:C\to\lim(D)$ is full.
\end{enumerate}
\end{lemma}
\begin{proof}
The data of $e,h,\kappa$ is equivalent to $e$, $h$, and $k:B\to \lim(D)$ such
that 
\tikzsetnextfilename{XaMiOOxI}
\begin{equation*}\begin{mytikzcd}
    A
        \ar[r,"h"]
        \ar[d,"e"']
\pnc	C
        \ar[d,"f"]
\\	B
\ar[r,"k"]
\pnc	\lim(D)
\end{mytikzcd}\end{equation*}
commutes, and $l:B\to C$ fills the latter square iff it fills all the squares
with the $D_j$.
\end{proof}
\begin{definition}\label{def:jfc}
A cone $\phi$ satisfying the conditions of the lemma is called \emph{jointly full}.
\end{definition}
\begin{remark}
The interest of this is that it allows us to talk about full `covers' of limits
without actually computing the limits, which is useful when talking about cones
and diagrams in the full subcategory of a clan-algebraic category on \z-extensions, which
does not admit limits.
\end{remark}
\begin{definition}\label{def:nice}
A \emph{nice diagram} in an \efcat $\efl$ is a $2$-truncated semi\-/simplicial
diagram
\begin{equation*}
\tikzsetnextfilename{YhTtFshV}
A_\bullet\;=\;\left(\begin{mytikzcd}[column sep = large]
    A_2     \ar[r, shift left = 3.2,"d_0" description]
            \ar[r, "d_1" description]
            \ar[r, shift right = 3.2, "d_2" description]
\pnc   A_1     \ar[r, shift left = 1.6, "d_0" description]
            \ar[r, shift right = 1.6, "d_1" description]
\pnc   A_0     \end{mytikzcd}\right)
\end{equation*}
where 
\begin{enumerate}
\item\label{def:nice:exts}
$A_0$, $A_1$, and $A_2$ are \z-extensions,
\item\label{def:nice:full}
the maps $d_0,d_1:A_1\to A_0$ are full,
\item\label{def:nice:atwo}
in the commutative square\tikzsetnextfilename{XqZklABG}
$\begin{mytikzcd}[row sep = small] A_2
        \ar[r,"d_0"']
        \ar[d,"d_2"']
\pnc 	A_1
        \ar[d,"d_1"]
\\	A_1
        \ar[r,"d_0"]
\pnc 	A_0 \end{mytikzcd}$the span constitutes a jointly full cone over the cospan, 
\vspace{-3mm}

\item\label{def:nice:sym}
there exists a `symmetry' map\tikzsetnextfilename{FvJsZmRM}
$\begin{mytikzcd}[row sep = small]
    A_1
        \ar[r,"d_1"]
        \ar[d,"d_0"']
        \ar[rd,"\sigma"]
\pnc 	A_0
\\	A_0
\pnc 	A_1
        \ar[l,"d_1"']
        \ar[u,"d_0"']
\end{mytikzcd}$making the triangles commute, and 
\item\label{def:nice:atilde}
there exists a \z-extension $\tilde A$ and full maps $f,g:\tilde A\epi
A_1$ constituting a jointly full cone over the
diagram 
\begin{equation*}
\tikzsetnextfilename{oMMQjvVA}
\begin{mytikzcd}[row sep = small]
    A_1
        \ar[rd,"d_1" very near start]
        \ar[d,"d_0"']
\pnc 	A_1
        \ar[dl,"d_0" very near end, crossing over]
        \ar[d,"d_1"]
\\	A_0
\pnc 	|[alias=B]|A_0
\end{mytikzcd}.
\end{equation*}
\end{enumerate}
\end{definition}
\begin{lemma}\label{lem:factor-nice} If $A_\bullet$ is a nice diagram in a 
clan-algebraic category $\efl$, the pairing $\angs{d_0,d_1}:A_1\to A_0\times
A_0$ factors as $A_1\xrightarrow{f}R\xrightarrow{\angs{r_0,r_1}}A_0\times A_0$,
where $f$ is full and $r=\angs{r_0,r_1}$ is monic and a componentwise full
equivalence relation.
\end{lemma}
\begin{proof}
Condition \ref{def:nice:atilde} of the preceding definition gives us the 
following diagram
\tikzsetnextfilename{LwuSZebT}
\begin{equation*}\begin{mytikzcd}
    \tilde A
        \ar[rrd, bend left=20,"g"', two heads]
        \ar[rdd, bend right=20,"f"', two heads]
        \ar[rd, "h"', two heads]
\\[-10pt]\pnc[-10pt] 
    |[label={[overlay, label distance=-3mm]-45:\lrcorner}]|S
\ar[r,"p"']
        \ar[d,"q"]
\pnc	
    A_1
        \ar[d,"\angs{d_0,d_1}"]
\\\pnc 
    A_1
        \ar[r,"\angs{d_0,d_1}"]
\pnc	
    A_0\times A_0
\end{mytikzcd},\end{equation*}
\ie $S$ is the kernel of $\angs{d_0,d_1}$ with projections $p,q$, $\tilde A$ is
a \z-extension, and $f,g,h$ are full. By right cancellation we deduce that $p$
and $q$ are full, and the existence of the factorization follows from
Lemma~\ref{lem:imfac}. Fullness of $r_0,r_1$ follows again from right
cancellation because $f$, $d_0$, and $d_1$ are full.

It remains to show that $r$ is an equivalence relation. This is easy: condition
4 gives symmetry, and condition 3 gives transitivity, and reflexivity follows
from the fact that $r_0$ admits a section as a full map into a \z-extension,
together with symmetry (we internalize the argument that if in a symmetric and
transitive relation everything is related to \emph{something}, then it is
reflexive.)
\end{proof}

\begin{definition} A \emph{\zext replacement} of an object $A$ in an \efcat
$\efl$ is a full map $f:\overline A\fullmap A$ from a \zext $\overline A$ to
$A$.
\end{definition} 
\zext replacements can always be obtained as $\wfsef$-factorizations of 
$0\to A$.
\begin{lemma}\label{lem:resolution}
For every object $A$ in an \efcat $\efl$ there exists a nice
diagram $A_\bullet$ with colimit $A$.
\end{lemma}
\begin{proof}
$A_0$ is constructed as a \zext replacement $f: A_0\fullmap A$ of $A$.
Similarly, $A_1$ is given by a \zext replacement $f_1:A_1\fullmap
A_0\times_AA_0$ of $A_0\times_AA_0$, and $A_2$ is a \zext replacement
$f_2:A_2\fullmap P$ of the pullback
\tikzsetnextfilename{ozfjApcR}
\begin{equation*}\begin{mytikzcd}
    |[label={[overlay, label distance=-3mm]-45:\lrcorner}]|P
        \ar[r,"p_0"]
        \ar[d,"p_1"']
\pnc	A_1
        \ar[d,"d_0"]
\\	A_1
        \ar[r,"d_1"]
\pnc	A_0
\end{mytikzcd}, 
\end{equation*}
with  $d_0,d_1,d_2:A_2\to A_1$ given by $d_0=p_0\circ f$, $d_2=p_1\circ f$, and
$d_1$ a lifting of $\angs{d_0\circ d_0,d_1\circ d_2}$ along $f_1$. The map
$\sigma$ is constructed as a lifting of the symmetry of $A_0\times_AA_0$ along
$f_1$. The object $\tilde A$ is a \zext replacement of the kernel of $f_1$.
\end{proof}

\begin{lemma}\label{lem:real-pres-nice}
For any clan-algebraic category $\efl$, the realization functor $J\reali$ preserves
jointly full cones in flat models, and nice diagrams.
\end{lemma}
\begin{proof}
The first claim follows since $J\reali$ is fully faithful on \zexts by
Lemma~\ref{lem:flat-junit-iso} and in both sides the weak factorization system
determined by the same generators. Thus there's a one-to-one correspondence
between lifting problems. The second claim follows since $J\reali$ preserves
\z-extensions and \z-extensions are flat by the fat small object argument.
\end{proof}

\begin{lemma}\label{lem:nerve-pres-quot-nice} For any clan-algebraic category
$\efl$, the functor $J\nerve:\efl\to\clopalg$ preserves quotients of nice
diagrams.
\end{lemma}
\begin{proof}
Given a nice diagram $A_\bullet$ in $\efl$, its colimit is the coequalizer of
$d_0,d_1:A_1\to A_0$. By Lemma~\ref{lem:factor-nice}, $\angs{d_0,d_1}$ factors
as $\angs{r_0,r_1}\circ f$ with $f$ full and $r=\angs{r_0,r_1}$ an equivalence
relation. The pairs $d_0,d_1$ and $r_0,r_1$ have the same coequalizer (since $f$
is epic), and $J\nerve$ preserves the coequalizer of $r_0,r_1$ since it preserves
full maps and kernel pairs. Finally, the coequalizer of $J\nerve(r_0),J\nerve(r_1)$ is
also the coequalizer of $J\nerve(d_0),J\nerve(d_1)$ since $J\nerve(f)$ is full and therefore
epic.
\end{proof}

\begin{theorem}\label{thm:clanalg-equiv}
If $\efl$ is clan-algebraic, then $J\nerve:\efl\to\czer{\efl}\op\hmod$ is an
equivalence. 
\end{theorem}
\begin{proof}
By density, $J\nerve$ is fully faithful. It remains to verify that it is
essentially surjective, and to this end we show that the unit map $\eta_A:A\to
J\nerve(J\reali(A))$ is an isomorphism for all $A\in\czer{\efl}\op\hmod$. Let
$A_\bullet$ be a nice diagram with colimit $A$. We have:
\begin{align*}
J\nerve(J\reali(A)) &= J\nerve(J\reali(\colim(A_\bullet)))\\
&=\colim(J\nerve(J\reali(A_\bullet)))&&\text{by Lemmas \ref{lem:real-pres-nice}
                                and \ref{lem:nerve-pres-quot-nice}}\\
&=\colim(A_\bullet) &&\text{by Lemma \ref{lem:flat-junit-iso}}\\
&=A
\end{align*}
\end{proof}
We are now ready to prove our main theorem.
\begin{theorem}\label{thm:main}
The contravariant biadjunction between clans and \efcats from
Proposition~\ref{prop:biadj} restricts to a contravariant biequivalence between
Cauchy-complete clans and clan-algebraic categories.
\begin{equation*}
\tikzsetnextfilename{k3eqwixu}
\begin{mytikzcd}
    \Clan\cc\op
        \ar[r,hook]
        \ar[d,leftrightarrow,"\simeq"']
\pnc\Clan\op
        \ar[d,"(-){\hmod}", shift left = 6pt]
        \ar[d,phantom,"\adj"]
        \ar[from=d,"\czer{-}\op", shift left = 6pt]
\\  \ClanAlg
        \ar[r,hook]
\pnc\EFCat
\end{mytikzcd}
\end{equation*}
\end{theorem}
\begin{proof}
By Theorem~\ref{thm:cc-equiv} the biadjunction is idempotent, and its
fixed points in $\Clan$ are Cauchy-complete clans. By
Theorems~\ref{thm:tmod-clanalg} and \ref{thm:clanalg-equiv}, the fixed points in
$\EFCat$ are precisely the clan-algebraic categories.
\end{proof}
\begin{remarks}\label{rems:char-algcat}\begin{enumerate}
\item\label{rems:char-algcat-arv} The characterization of \efcats of the form
$\tmod$ in terms of conditions \ref{def:cac:d}, \ref{def:cac:cg}, and
\ref{def:cac:fq} generalizes, and is in fact inspired by,
Adámek--Rosický--Vitale's characterization of \emph{algebraic categories} (i.e.\
categories of models of algebraic theories) as \emph{Barr-exact cocomplete
locally small  categories with a strong generator of compact regular
projectives} \cite[Corollary~18.4]{adamek2010algebraic}\footnote{See
also~\cite[Theorem~5.5]{adamek2004quasivarieties} for an earlier version of this
using slightly different notions}. In particular, if $\clant$ is a
finite-product clan then the regular projectives in $\tmod$ are precisely the
$0$-extensions, thus the strong generation requirement corresponds to
\ref{def:cac:d} (which we have elected to state in terms of density). The
Barr-exactness requirement refines to the existence of \emph{full coequalizers
of componentwise full equivalence relations} in the clanic case. The fullness
requirements are void in the finite-product case, since coequalizers as well as
split epimorphisms are always regular epimorphisms.
\item
As emphasized by the referee, Definition~\ref{def:nice} of `nice diagrams'
resembles Carboni--Vitale's definition of \emph{pseudo-equivalence
relations}~\cite[Def.~6]{carboni1998regular}. Indeed, the resolution by nice
diagrams given in Lemma~\ref{lem:resolution} reduces to pseudo-equivalence
relations in projective objects in the case of finite-product clans, which are
related to the fact that algebraic categories are \emph{ex/wlex-completions} in
the sense of Carboni and Vitale. An analogous presentation of clan-algebraic
categories as a completion of categories of $0$-extensions (suitably
axiomatized) is subject of future work.
\end{enumerate}

\end{remarks}

\subsection{Clan-algebraic weak factorization systems on
\texorpdfstring{$\Cat$}{Cat}}\label{suse:caws-cat}

The characterization of \efcats of models of clans as clan-algebraic categories
allows to exhibit new clans by defining suitable w.f.s.s on l.f.p.\ categories.
In this subsection we demonstrate this by defining three more clan-algebraic
w.f.s.s on $\Cat$.

We start with the clan-algebraic w.f.s.\ from
Example~\ref{ex:wfsef}\ref{ex:wfsef:cat}, corresponding to the `standard'
clan-presentation $\clant_\Cat$ from \Subsection\ref{suse:cat-clan}. We observed
that this w.f.s.\ is cofibrantly generated by the functors $0\incl 1$ and
$2\incl\bbtwo$. Our strategy to define new clan-algebraic w.f.s.s is to add
additional generators. If we make sure that the domain and codomain of these are
compact $0$-extensions, we only have to verify condition~\ref{def:cac:fq} when
verifying that the new w.f.s.\ is still clan-algebraic. The additional
generators we consider are the arrows $\PP\to\bbtwo$ and $2\to 1$, where $\PP$
is the `parallel pair category' $\bullet\rightrightarrows\bullet$. By adding
either one or both of the additional generators we obtain three additional
w.f.s.s $(\classe\sO,\classf\sO)$, $(\classe\sA,\classf\sA)$, and
$(\classe\sOA,\classf\sOA)$, where:
\begin{align*}
  \classf       &=\braces{(0\to 1),(2\to\bbtwo)}^\pitchfork
\\\classf\sO &=\braces{(0\to 1),(2\to\bbtwo),(2\to 1)}^\pitchfork
\\\classf\sA &=\braces{(0\to 1),(2\to\bbtwo),\phantom{(2\to 1),}(\PP\to\bbtwo)}^\pitchfork
\\\classf\sOA   &=\braces{(0\to 1),(2\to\bbtwo),{(2\to 1),}(\PP\to\bbtwo)}^\pitchfork
\end{align*}
We have already observed that $\classf$ consists of the functors that are full
and surjective on objects, and it is easy to see that $\classf\sO$ contains only
those functors which are full and \emph{bijective} on objects, whereas
$\classf\sA$ consists of functors which are \emph{fully faithful} and surjective
on objects. Finally, $\classf\sOA$ only contains functors which are fully
faithful and bijective on objects, i.e.\  isomorphisms of categories.

To convince ourselves that the new w.f.s.s are indeed clan-algebraic we only
have to verify that for every equivalence relation $\angs{p,q}:\RR\to\AA\x\AA$
in $\Cat$, the coequalizer is in either of $\classf\sO,\classf\sA,\classf\sOA$
whenever $p$ and $q$ are, since effectivity has already been established for
equivalence relations with components in $\classf$. This is not difficult to see
for $\classf\sO$ and $\classf\sA$, and trivial for $\classf\sOA$.

The coclans corresponding to the new w.f.s.s are:
\begin{itemize}
\item $\clant_{\Cat_{\sO}}\op= \bracetext{categories free on finite graphs}$,
with functors $G^*\to H^*$ arising from faithful graph morphisms as \codisplay
maps,
\item $\clant_{\Cat_{\sA}}\op = \bracetext{finitely presented categories}$, with
injective-on-objects functors as \codisplay maps, and 
\item $\clant_{\Cat_{\sOA}}\op = \bracetext{finitely presented categories}$, with
arbitrary functors as \codisplay maps.
\end{itemize}
We note the clan $\clant_{\Cat_{\sOA}}$ is simply the finite-limit theory of
categories. One may ask whether the clans $\clant_{\Cat_{\sO}}$,
$\clant_{\Cat_{\sA}}$, and $\clant_{\Cat_{\sOA}}$ admit simple syntactic
presentations by GATs, and indeed they do. To obtain such a presentation e.g.\
for $\clant_{\Cat_{\sO}}$, we have to modify the GAT $\TT_\Cat$ in such a way
that the syntactic category stays the same, but acquires additional display
maps, such as the diagonal $(x:O)\to(x\,y:O)$ corresponding to the new generator
$2\to 1$. Display maps in the syntactic category of a GAT are always of the form
$p\circ i$ where $p$ is a projection omitting a finite number of variables and
$i$ is an isomorphism (Proposition~\ref{prop:clan-from-cxcat}), so to turn
$(x:O)\to(x\,y:O)$ into a display map we have to make $(x:O)$ isomorphic to an
\emph{extension} of $(x\,y:O)$. To achieve this we postulate a new type family
$x\,y:O\ent E_O(x,y)$ and add axioms forcing the projection $(x\,y:A\sco
z:E_O(x,y))\to (x:A)$ to become an isomorphism:
\begin{equation}\label{eq:o-id}
\begin{aligned}
    x\,y:O                & \sent E_O(x,y)
\\  x:O                   & \sent r_O(x):E_O(x,x) 
\\  x\,y:O\,,\,p:E_O(x,y) & \sent x=y
\\  x:O\,,\,p:E_O(x,x)    & \sent r_O(x)=p
\end{aligned}
\end{equation}
The function symbol $r_O$ gives a section for the projection, and the two last
axioms force the retract to be an isomorphism. We recognize at once that these
axioms make $E_O$ an \emph{extensional identity type} of
$O$~\cite[Section~3.2]{hofmann1997extensional}: the term $r_O$ is reflexivity,
and the third and fourth rule give equality reflection and uniqueness of
identity proofs. We write $\TT_{\Cat_{\sO}}$ for the extension of the GAT
$\TT_{\Cat}$ by the axioms~\eqref{eq:o-id}, and $\Cat_{\sO}$ for the
corresponding clan-algebraic category.

Similarly, we obtain a GAT-representation $\TT_{\Cat_{\sA}}$ of the clan
$\clant_{\Cat_{\sA}}$ by augmenting $\TT_\Cat$ by a type family $E_A$ with the
following rules:
\begin{equation}\label{eq:a-id}
\begin{aligned}
   x\,y:O\,,\,f\,g:A(x,y)                &\sent E_A(f,g)
\\ x\,y:O\,,\,f:A(x,y)                   &\sent r_A(f):E_A(f,f) 
\\ x\,y:O\,,\,f\,g:A(x,y)\,,\,p:E_A(f,g) &\sent f=g
\\ x\,y:O\,,\,f:A(x,y)\,,\,p:E_A(f,f)    &\sent r_A(f)=p
\end{aligned}
\end{equation}
Adding both sets of axioms \eqref{eq:o-id} and \eqref{eq:a-id} to $\TT_\Cat$
yields a GAT for the clan $\clant_{\Cat_{\sOA}}$, i.e.\ the finite-limit theory
of categories.
\begin{remark}
As pointed out by the referee, an enriched version of the
$(\classe\sA,\classf\sA)$-instance of the exactness condition \ref{def:cac:fq}
appears under the name \emph{ff-exactness} in Bourke and
Garner's~\cite{bourke2014two}. Further investigation may be needed.
\end{remark}

\section{A counterexample}\label{se:counterexample}

This section gives a common counterexample to two related natural questions
about the extension--full w.f.s.\ on a clan-algebraic category $\efl$:
\begin{enumerate}
\item Does every compact object admit a full map from a compact $0$-extension?
\item Does the weak factorization system always restrict to compact
$0$-extensions? 
\end{enumerate}
The counterexample to both question is given by the category of models of the
following GAT with infinitely many sorts and operations:
\begin{alignat*}{3}
&\sent X\\
&\sent Y\\
y{:}Y&\sent Z_n(y)&n\in\N\\
x{:}X&\sent f(x):Y\\
x{:}X&\sent g_n(x):Z_n(f(x))\qquad&n\in\N
\end{alignat*}
Its category of models is equivalent to the set-valued functors on the posetal
category
\begin{equation*}
\CC = \left(
\tikzsetnextfilename{UnySuCzy}
\begin{mytikzcd}[column sep = 5pt]
    X       \ar[dd, bend right = 40, "f"']
            \ar[d, "g_0" near end]
            \ar[dr, "g_1" very near end]
            \ar[drrr, "g_n" near end]
\\  Z_0     \ar[d, -{Triangle[open]},"z_0"]
\pnc   Z_1     \ar[dl, -{Triangle[open]},"z_1"]
\pnc   \dots
\pnc   Z_n     \ar[dlll, -{Triangle[open]},"z_n"]
\pnc   \dots
\\  Y
\end{mytikzcd}
\right)
\end{equation*}
and the w.f.s.\ on $[\CC,\Set]$ is generated by the arrows $(\varnothing\codito
\yo(X))$, $(\varnothing\codito\yo(Y))$, and $(\yo(Y)\codito\yo(Z_n))$ for
$n\in\N$, reflecting the idea that models $A:\CC\to\Set$ can be built up by
successively adding elements to $A(X)$ or $A(Y)$, and to $A(Z_n)$ over a given
element $x$ of $A(Y)$, as in the following pushouts.
\begin{equation*}
\tikzsetnextfilename{HJSZMZGG}
\begin{mytikzcd}[sep = small]
    {\varnothing}
        \ar[r]
        \ar[d,{Triangle[open,reversed]->}]
\pnc	{A}
        \ar[d,{Triangle[open,reversed]->}]
\\	{1=\yo(X)}
        \ar[r]
\pnc	|[label={[overlay, label distance=-3mm]135:\ulcorner}]|{1+A}
\end{mytikzcd}
\qquad
\tikzsetnextfilename{RcKjobXg}
\begin{mytikzcd}[sep = small]
    \varnothing
        \ar[r]
        \ar[d,{Triangle[open,reversed]->}]
\pnc	A
        \ar[d,{Triangle[open,reversed]->}]
\\	{\yo(Y)}
        \ar[r,""]
\pnc	|[label={[overlay, label distance=-3mm]135:\ulcorner}]|{\yo(Y)+A}
\end{mytikzcd}
\qquad
\tikzsetnextfilename{gGKkVdXC}
\begin{mytikzcd}[sep = small]
    {\yo(Y)}
        \ar[r,"\hat x"]
        \ar[d,{Triangle[open,reversed]->}]
\pnc	{A}
        \ar[d,{Triangle[open,reversed]->}]
\\	{\yo(Z_n)}
        \ar[r]
\pnc	|[label={[overlay, label distance=-3mm]135:\ulcorner}]|{B}
\end{mytikzcd}
\end{equation*}
The following lemma gives explicit descriptions of the w.f.s.\ and the compact
objects in $[\CC,\Set]$.
\begin{lemma}\label{lem:xyz-gat} Let $f:A\to B$ in $[\CC,\Set]$.
\begin{enumerate}[label=(\roman*)]
\item $f$ is full if and only if it is componentwise surjective and the
naturality squares for $z_n$
are weak pullbacks for all $n\in\N$.
\item $f$ is an extension if an only if $f_X:A(X)\to B(X)$ is injective, and the
squares
\begin{equation*}
\tikzsetnextfilename{QHYOqPve}
\begin{mytikzcd}[sep = small]
    A(X)
        \ar[r,""']
        \ar[d,""]
\pnc	B(X)
        \ar[d,""]
\\	A(Y)
        \ar[r,""]
\pnc	B(Y)
\end{mytikzcd}\qquad\qquad
\tikzsetnextfilename{XgILxNQm}
\begin{mytikzcd}[sep = small]
    A(X)
        \ar[r,""']
        \ar[d,""]
\pnc	B(X)
        \ar[d,""]
\\	A(Z_n)
        \ar[r,""]
\pnc	B(Z_n)
\end{mytikzcd}
\end{equation*}
are \emph{quasi-pushouts}, in the sense that the gap maps $A(Y)+_{A(X)}B(X)\to
B(Y)$ and $A(Z_n)+_{A(X)}B(X)\to B(Z_n)$ are injective. (This implies that the
components $f_{Y}$ and $f_{Z_n}$ are also injective).
\item $A$ is a $0$-extension if an only if $A(f)$ and all $A(g_n)$ are
injective.
\item\label{lem:xyz-gat-comp} $A$ is compact if an only if (a) it is
componentwise finite, and (b) $A({f_n}):A(X)\to A(Z_n)$ is a bijection for all
but finitely many $n\in\N$.
\end{enumerate}
\qed
\end{lemma}
Using this lemma, we can give negative answers to the two question at the
beginning of the section.
\begin{proposition}
\begin{inumerate}
\item The object $P$ in the pushout
\begin{equation*}
\tikzsetnextfilename{WZbwZkuj}
\begin{mytikzcd}[sep = small]
    {\yo(Y)+\yo(Y) }  \ar[r]
                    \ar[d] 
\pnc   {\yo(X)+\yo(X)}   \ar[d]
\\  {\yo(Y)}          \ar[r]
\pnc   |[label={[overlay, label distance=-3mm]135:\ulcorner}]| P
\end{mytikzcd}
\end{equation*}
is compact, but does not admit a full map from a compact $0$-extension.
\item The map $\yo(Y)\to\yo(X)$ does not admit an extension--full factorization
through a compact object.
\end{inumerate}
\end{proposition}
\begin{proof}
For the first claim, $P$ is compact as a finite colimit of representables. Let $f:E\fullmap P$ be a full map with $E$ a $0$-extension. For each $n\in\N$ we
get a diagram
\tikzsetnextfilename{dvOZrkaD}
\begin{equation*}\begin{mytikzcd}[column sep = small]
            E(X)
                \ar[ddr,tail,"E(f)"']
                \ar[dr,tail,"E(g_n)" very near end]
                \ar[rr,"f_X", two heads]
\pnc\pnc          2
                \ar[ddr]
                \ar[dr]
\\[-10pt]\pnc  E(Z_n)
                \ar[rr, crossing over,"f_{Z_n}" near start, two heads]
                \ar[d,"E(z_n)"]
\pnc\pnc[10pt]    2
                \ar[d,""]
\\\pnc         E(Y)
                \ar[rr,"f_Y",two heads]
\pnc\pnc          1
\end{mytikzcd}\end{equation*}
where $E(f)$ and $E(g_n)$ are injective because $E$ is a $0$-extension, and the
components of $f$ are surjective and the $z_n$-naturality square is a weak
pullback since because $f$ is full. In particular $E(Y)$ is inhabited and the
fibers of $E(z_n)$ have at least two elements. Since the fibers of $E(f)$ have
at most one element, this means that $E(g_n)$ can't be surjective for any $n$,
and it follows from Lemma~\ref{lem:xyz-gat}\ref{lem:xyz-gat-comp} that $E$ is
not compact.

For the second claim consider an extension $e:\yo(Y)\to A$ such that
$A\to\yo(X)=1$ is full. Then $A(Y)$ is inhabited, all $A(Z_n)\to A(Y)$ are
surjective, and $1+A(X)\to A(Y)$ is injective. From this we can again deduce
that none of the $A(g_n)$ are surjective and thus $A$ is not compact.
\end{proof}

\section{Models in higher types}\label{se:higher-types}

One practical use of having inequivalent clans with equivalent categories of
$\Set$-models is that they can have inequivalent $\infty$-categories of models
in the $\infty$-category $\Spaces$ of \emph{homotopy types} (a.k.a.\ `spaces').
We leave this issue for future work and content ourselves here with outlining
some main ideas.

Reasoning informally, in the following we identify sets with discrete
($0$-truncated) homotopy types, and $1$-categories with $\infty$-categories all
of whose $\hom$-types are $0$-truncated. Given a clan $\clant$ it then makes
sense to define an \emph{$\infty$-model} to be an $\infty$-functor
$\clant\to\Spaces$ into spaces which preserves the terminal object and pullbacks
of display maps in the suitable $\infty$-categorical sense, and to write
$\tmod_\infty\subs[\clant,\Spaces]$ for the $\infty$-category of models, as a
full $\infty$-subcategory of the functor-$\infty$-category $[\clant,\Spaces]$.

The first observation is that for $\clant$ a \emph{finite-limit clan}, the
$\infty$-category $\tmod_\infty$ of $\infty$-models is in fact equivalent to
(the $\infty$-category corresponding to) the \emph{$1$-category} $\tmod$ of
$1$-models!
This is because finite $\infty$-limit preserving $\infty$-functors preserve
truncation levels, and thus every finite $\infty$-limit preserving
$F:\mcl\to\Spaces$ must factor through the inclusion of $\Set\incl\Spaces$,
since all objects of $1$-categories are (representably) $0$-truncated.

For finite-product clans, on the other hand, there is no such restriction. The
\oo-models of the finite-product clan $\mcc_\Mon$ of monoids, for example, are
the models of the \emph{associative
\oo-operad}~\cite[Section~4.11]{lurie2017ha}, whereas the \oo-models of the
finite-product theory of abelian groups are related to the Dold--Kan
correspondence. Variants of this phenomenon are discussed under the name
`animation' in~\cite{cesnavicius2024purity};
Rosický's~\cite{rosicky2007homotopy} contains an earlier account.

Now the nice thing about clans is that they admit finer graduations of `levels
of strictness'. Among the clans $\clant$, $\clant_{\Cat_{\sO}}$,
$\clant_{\Cat_{\sA}}$, and $\clant_{\Cat_{\sOA}}$ from
\Section\ref{suse:caws-cat}, for example, we know that the \oo-models of the
finite-limit clan $\clant_{\Cat_{\sOA}}$ are precisely the strict
$1$-categories. The presence of the extensional identity type on $O$ in
$\clant_{\Cat_{\sO}}$ behaves like a kind of `partial finite-limit completion',
and has the effect that the sort $O$ is interpreted by a $0$-type in every model
$\clant_{\Cat_{\sO}}\to\Spaces$, whereas the presence of extensional identities
on $A$ in $\clant_{\Cat_{\sA}}$ has the effect that the projection
$(x\,y:O\,,\,f:A(x,y))\to(x\,y:O)$ is mapped to a function with $0$-truncated
fibers by every \oo-model $C:\clant_{\Cat_{\sA}}\to\Spaces$. This means that
\oo-models of $\clant_{\Cat_{\sA}}$ are \emph{pre-categories} in the sense of
\emph{Homotopy Type Theory} \cite[Definition~9.1.1]{hottbook}, whereas
\oo-models of $\clant_{\Cat_{\sO}}$ seem to correspond to
\emph{Segal-categories}~\cite[Section~2]{hirschowitz1998descente},
\cite[Section~5]{bergner2010survey}. Finally, the clan $\clant_\Cat$ does not
impose \emph{any} truncation conditions, which makes its \oo-models resemble
\emph{Segal spaces} (not necessarily complete), in the sense
of~\cite[Section~4]{rezk2001model}.
\begin{remark}\label{rem:univalence} Notably absent from the list of higher
algebraic structures represented by the variants of $\clant_\Cat$ are univalent
$\infty$-categories. In fact, the requirement on a clan $\clant$ that
$\tmod_\infty\simeq\infty\h\Cat$ is incompatible with the requirement that
$\tmod\simeq\Cat$, since $\tmod$ is the full subcategory of $\tmod_\infty$ on
$0$-truncated objects, and the only $0$-truncated univalent $1$-categories are
the \emph{rigid} ones, i.e.\  those whose with only trivial automorphisms. It
seems unlikely to the author that $\infty\h\Cat$ can be described by a
`$1$-clan'. See~\cite{ahrens2020higher,ahrens2021univalence} for recent work
addressing this univalence issue, using techniques building on Makkai's
\emph{First order logic with dependent sorts (FOLDS)}~\cite{makkai1995first}.
Another recent work connecting Cartmell's GATs and Makkai's FOLDS is Chaitanya
Leena Subramaniam's thesis~\cite{leena2021dependent}.
\end{remark}

\appendix

\section{Locally finitely presentable categories, weak
factorization systems, and Quillen's small object argument}\label{se:lfp-soa}

This appendix recalls basic definitions and facts about the concepts mentioned
in the title.

\begin{definition}
A category $\mcc$ is called \emph{filtered}, if every diagram 
$D:\JJ\to\mcc$ with finite domain admits a cocone. A \emph{filtered colimit}
is a colimit of a diagram indexed by a filtered category.
\end{definition}
\begin{definition}\label{def:compact-lfp}
Let $\mfx$ be a cocomplete  locally small category.
\begin{enumerate}
\item\label{def:compact-lfp:compact}
An object $C\in\mfx$ is called \emph{compact}, if the covariant 
representable functor 
\begin{equation*}
\mfx(C,-):\mfx\to\Set
\end{equation*}
preserves small filtered colimits.
\item\label{def:compact-lfp:lfp}
$\mfx$ is called \emph{locally finitely presentable} (\lfp) if it admits a
\emph{small dense family of compact objects}, i.e.\ a family $\fami{C_i}{i\in
I}$ of compact objects indexed by a small set $I$, such that the nerve functor 
\vspace{-2ex}
\begin{equation*}
J\nerve:\mfx\to\psh\CC
\end{equation*}
of the inclusion $J:\CC\incl\mfx$ of the full subcategory on the
$\fami{C_i}{i\in I}$ is fully faithful.
\end{enumerate}
\end{definition}
\begin{remarks}
\begin{enumerate}
\item Compact objects are also known as \emph{finitely presentable objects},
e.g.\ in~\cite{gabriel1971lokal,adamek1994locally}. We adopted the term compact
from~\cite[Definition~A.1.1.1]{lurie2009higher} since it is more concise, and in
particular since \emph{compact \zext} sounds less awkward than \emph{finitely
presented \zext}. Moreover I think the fact that objects of algebraic categories
(such as groups, rings, modules \dots) are compact if and only if they admit a
presentation by finitely many generators and relations is an important theorem,
which is difficult to state if one uses the same terminology for the syntactic
and the categorical notion.
\item The density condition in the definition is equivalent to saying that the
family $\fami{C_i}\ini$ is a \emph{strong generator}, in the sense that the
canonical arrow
\begin{equation*}
\coprod_{i\in I,f:C_i\to A}\!\!\!C_i\to A
\end{equation*}
is an extremal epimorphism for all $A\in \mfx$. We stated the definition in
terms of density here, since nerve functors play a central role in this work,
contrary to strong generation.
\item The notion of \lfp category is a special case of the notion of
\emph{locally $\alpha$-presentable category} for a regular cardinal
$\alpha$~\cite{gabriel1971lokal,adamek1994locally}. In this work, only the case
$\alpha=\omega$ plays a role.
\end{enumerate}
\end{remarks}
\begin{definition}\label{def:wfs}
Let $\mcc$ be a category.
\begin{enumerate}
\item\label{def:wfs:llp-rlp}
Given two arrows $f:A\to B$, $g:X\to Y$ in $\mcc$, we say that $f$ has the
\emph{left lifting property} (\llp) \wrt $g$ (or equivalently that $g$ has the
\emph{right lifting property} (\rlp) \wrt $f$), and write $f\pitchfork g$, if in
each commutative square
\begin{equation*}
\tikzsetnextfilename{dqHrPkUH}
\begin{mytikzcd}
    A
        \ar[r,"h"]
        \ar[d,"f"']
\pnc   X
        \ar[d,"g"]
\\  B
        \ar[r,"k"]
        \ar[ru, dashed, "m"]
\pnc   Y
\end{mytikzcd}
\end{equation*}
there exists a diagonal arrow $h$ making the two triangles commute.
\item\label{def:wfs:orth-class}
Given a class $\classe\subs{\mor(\mcc)}$ of arrows in $\mcc$, we define:
\begin{align*}
{}^\pitchfork\classe&=\setof{f\in\mor(\mcc)}
        {\fa g\in\classe\qdot f\pitchfork g}\\
\classe^\pitchfork&=\setof{g\in\mor(\mcc)}
        {\fa f\in\classe\qdot f\pitchfork g}
\end{align*}
\item\label{def:wfs:wfs} A \emph{weak factorization system} (\wfs) on $\mcc$ is
a pair $\classl,\classr\subs\mor(\mcc)$ of classes of morphisms such that
$\classl^\pitchfork = \classr$, $\classr^\pitchfork = \classl$, and every
$f:A\to B$ in $\mcc$ admits a factorization $f=r\circ l$ with $l\in\classl$ and
$r\in\classr$.
\end{enumerate}
\end{definition}
We call $\classl$ the \emph{left class}, and $\classr$ the \emph{right class} of
the \wfs. One can show that left classes of w.f.s.s contain all isomorphisms, 
and are closed under composition and pushouts, i.e.\ if 
\begin{equation*}
\tikzsetnextfilename{ZrFUwxmz}
\begin{mytikzcd}
    A
        \ar[r]
        \ar[d,"l"']
\pnc   B
        \ar[d,"m"]
\\  C
        \ar[r]
\pnc   |[label={[overlay, label distance=-3mm]135:\ulcorner}]|D
\end{mytikzcd}
\end{equation*}
is a pushout in $\mcc$ and is a left map, then so is $m$. Dually, right maps are
closed under (isomorphisms, composition, and) pullbacks. With this, we have the
prerequisites to state Quillen's \emph{small object argument}.
\begin{theorem}[Small object argument for \lfp
categories]\label{thm:soa} Let $\classe\subs\mor(\mfx)$ be a small set of
morphisms in a l.f.p.\ category. Then
$({}^\pitchfork(\classe^\pitchfork),\classe^\pitchfork)$ is a w.f.s.\ on $\mfx$.
\end{theorem}
\begin{proof}
Hovey~\cite[Thm.~2.1.14]{hovey1999model} and
Riehl~\cite[Thm.~12.2.2]{riehl2014categorical} prove stronger statements in a
more general setting.
\end{proof}

\section{Generalized algebraic theories}\label{se:gats}

Cartmell's \emph{generalized algebraic theories} extend the notion of algebraic
theory (which can be `single sorted', such as the theories of \emph{groups} or
\emph{rings}, or `many sorted', such as the theories of \emph{reflexive graphs},
\emph{chain complexes of abelian groups}, or \emph{modules over a non-fixed base
ring}) by introducing \emph{dependent sorts} (a.k.a.\ dependent `types'), which
represent families of sets and can be used e.g.\ to axiomatize the notion of a
(small) category $\CC$ as a structure with a set $\CC_0$ of objects, and a
family $(\CC(A,B))_{A,B\in \CC_0}$ of $\hom$-sets (see \eqref{eq:gat-cat}
below).

Compared to ordinary algebraic theories, whose specification in terms of sorts,
operations, and equations is fairly straightforward, the syntactic description of generalized algebraic theories is complicated by
the fact that the domains of definition of operations and dependent sorts, and
the codomains of operations, may themselves be compound expressions involving
previously declared operations and sorts, whose well-formedness has to be
ensured and may even depend on the equations of the theory. This means 
that we have to state the declarations of \emph{sorts} and of \emph{operations},
and the \emph{equations} (which we collectively refer to as \emph{axioms} of the
theory) in an ordered way, where the later axioms have to be well-formed on the
basis of the earlier axioms. This looks as follows in the case of the
generalized algebraic theory $\TT_\Cat$ of categories:
\begin{equation}\label{eq:gat-cat}
\setlength{\jot}{1pt}
\begin{aligned}
        {}
&       \sent O
\\      x\, y : O 
&       \sent A\dep{x,y}
\\      x:O
&       \sent \id\dep{x}:A\dep{x,x}
\\      x\,y\,z:O \sco f:A\dep{x,y} \sco g:A\dep{y,z}
&       \sent g{\circ} f : A\dep{x,z}
\\      x\,y:O\sco f: A\dep{x,y}
&       \sent \id\dep{y}{\circ} f = f 
\\      x\,y:O\sco f: A\dep{x,y}
&       \sent  f{\circ} \id\dep{x} = f 
\\      w\,x\,y\,z:O\sco e:A\dep{w,x}\sco f:A\dep{x,y}\sco g:A\dep{y,z} 
&       \sent (g{\circ}f){\circ}e=g{\circ}(f{\circ}e)
\end{aligned}
\end{equation}
Each line contains one axiom, the first two declaring the sort $O$ of objects
and the dependent sort $A\dep{x,y}$ of arrows, the third and the fourth declaring
the identity and composition operations, and the last three stating the identity
and associativity axioms.

Each axiom is of the form $\Gamma\ent\mcj$, where the $\mcj$ on the right of the
`turnstile' symbol `$\ent$' is the actual declaration or equation, and the part
$\Gamma$ on the left---called `context'---specifies the sorts of the variables
occurring in $\mcj$. Note that the ordering of these `variable declarations' is
not arbitrary, since the sorts of variables may themselves contain variables
which have to be declared further left in the context. An example is the context
$(x\,y\,z:O \sco f:A\dep{x,y}\sco g:A\dep{y,z})$ of the composition operation,
where the sorts of the `arrow' variables $f,g$ depend on the `object' variables
$x,y,z$. See Figure \ref{fig:gat-gring} for another example generalized
algebraic theory: the generalized algebraic theory of \emph{rings graded over
monoids}.
\begin{figure}
\setlength{\jot}{0pt}
\begin{equation*}\begin{aligned}
\toprule
& 
&&\sent M \\
&& u : M&\sent R\dep{u} \\
&
&&\sent e:M \\
&& u\,v:M &\sent u\ap v: M \\
&& u:M &\sent 0\dep{u}: R\dep{u}\\
&& u:M\sco x\,y:R\dep{u}&\sent x+y : R\dep{u}\\
&& u:M\sco x:R\dep{u}&\sent - x : R\dep{u}\\
&& &\sent 1:R\dep{e}\\
&& u\,v:M\sco x:R\dep{u}\sco y:R\dep{v}&\sent x\ap y:R\dep{u\ap v}\\
&
&u:M&\sent e\ap u = u = u\ap e\\
&&u\,v\,w:M&\sent (u\ap v)\ap w = u\ap (v\ap w)\\
&&u:M\sco x\,y:R\dep{u}&\sent x+y = y+x\\
&&u:M\sco x\,y:R\dep{u}&\sent x+0\dep{u} = x\\
&&u:M\sco x\,y:R\dep{u}&\sent x+(-x) = 0\dep{u}\\
&&u:M\sco x:R\dep{u}&\sent 1\ap x = x = x\ap 1\\
&&u\,v\,w:M\sco x:R\dep{u}\sco y:R\dep{v}\sco z:R\dep{w}&\sent(x\ap 
                                                        y)\ap z=x\ap(y\ap z)\\
&&u\,v:M\sco x:R\dep{u}\sco y\,z: R\dep{v}&\sent x\ap(y+z)=x\ap y + x\ap z\\
&&u\,v:M\sco x\,y:R\dep{u}\sco z: R\dep{v}&\sent (x+y)\ap z = x\ap z + y \ap z
\\\bottomrule
\end{aligned}\end{equation*}

\vspace{-\baselineskip}
\caption{The generalized algebraic theory of monoid-graded
rings}\label{fig:gat-gring}
\end{figure}

The dependent structure of contexts and the  well-formedness requirement of
axioms on the basis of other axioms makes the formulation of a general notion of
generalized algebraic theory somewhat subtle and technical. We refer to
\cite{cartmell1978generalised,cartmell1986generalised} for the authoritative
account and to \cite[Section 6]{pitts2000categorical} and \cite[Section
2]{garner2015combinatorial} for rigorous and concise summaries. The good news is
that to understand specific examples of GATs, these technicalities may safely be
ignored: all we have to know is that for every generalized algebraic theory
$\TT$ there is a notion of `derivable judgment' which includes the axioms and is
closed under various \emph{rules} expressing that the set of derivable judgments
is closed under operations like substitutions and weakening, and that equality
is reflexive, symmetric, and transitive. 

Besides the forms of judgments
\begin{align*}
\Gamma&\ent S &&\text{`$S$ is a sort in context $\Gamma$'}\\
\Gamma&\ent t:S &&\text{`$t$ is term of sort $S$ in context $\Gamma$'}\\
\Gamma&\ent s=t:S &&\text{`$s$ and $t$ are equal terms in context $\Gamma$'}\\
\intertext{that we have already encountered, 
we consider the following additional forms of judgments:}
\Gamma&\ent S = T
    &&\text{`$S$ and $T$ are equal sorts in context $\Gamma$'}\\
\Gamma&\ent 
    &&\text{`$\Gamma$ is a context'}\\
\Gamma=\Delta&\ent 
    &&\text{`$\Gamma$ and $\Delta$ are equal contexts'}\\
\Gamma&\ent \sigma:\Delta 
    &&\text{`$\sigma$ is a {substitution} from $\Gamma$ to $\Delta$'}\\
\Gamma&\ent \sigma=\tau:\Delta 
    &&\text{`$\sigma$ and $\tau$ are equal {substitutions} from $\Gamma$ to 
            $\Delta$'}
\end{align*}
The last two of these introduce a novel kind of expression called
\emph{substitution}: a {substitution} $\Gamma\ent\sigma:\Delta$ is a list of
terms that is suitable to be simultaneously substituted for the variables in a
judgment in context $\Delta$ (in particular $\sigma$ and $\Delta$ must have the
same length), to produce a new judgment in context $\Gamma$, as expressed by the
following \emph{substitution} rule.
\begin{equation*}
\ax{\Gamma\sent\sigma:\Delta}
\ax{\Delta\sent\mcj}
\bi{\Gamma\sent\mcj[\sigma]}
\drap
\tag{Subst}\label{eq:rule-subst}
\end{equation*}
Here, $\mcj[\sigma]$ is the result of simultaneous substitution of the terms in
$\sigma$ for the variables in $\mcj$, replacing each occurrence of the  $i$th
variable declared in $\Delta$ with the $i$th term in~$\sigma$. This operation of
simultaneous substitution also appears in the derivation rules for substitutions
themselves, which we present in the following table together with the rules for
the formation of well-formed contexts:
\begin{equation}\label{eq:context-substitution-rules}
\begin{array}{c@{\qquad\qquad}c}
            \ax{\phantom,}
            \ui{\ecxt\sent}
            \drap
&           \ax{\phantom,}
            \ui{\Gamma\sent():\ecxt}
            \drap
\\[15pt]    \ax{\Gamma\sent A}
            \ui{\Gamma,y:A\ent}
            \drap
&           \ax{\Gamma\sent\sigma:\Delta}
            \ax{\Delta\sent A}
            \ax{\Gamma\sent t:A[\sigma]}
            \ti{\Gamma\sent(\sigma,t):(\Delta,x:A)}
            \drap
\end{array}
\end{equation}
The two rules in the first line say respectively that the \emph{empty context}
$\varnothing$ is a context, and that for any context $\Gamma$, the \emph{empty
substitution} $()$ is a substitution to the empty context. The first rule in the
second line is known as \emph{context extension}, since it says that we can
extend any context by a well-formed sort in this context (here $y$ has to be a
`fresh' variable, i.e.\ a variable not appearing in $\Gamma$). The last rule
says that a substitution to an extended context is a pair of a substitution into
the original context and a term whose sort is a \emph{substitution instance} of
the extending sort---it wouldn't make sense to ask for $t$ to be of sort $A$
since $A$ is only well-formed in context $\Delta$, and we want something in
context $\Gamma$.

\subsection{The syntactic category of a generalized algebraic theory}
\label{suse:gat-syncat}
\begin{definition}\label{def:syncat}
The \emph{syntactic category} $\mcc[\TT]$ of a generalized algebraic theory
$\TT$ is given as follows.
\begin{itemize}
\item The objects are the contexts of $\Gamma$ \emph{modulo derivable equality},
\ie contexts $\Gamma$ and $\Delta$ are identified if the judgment
$\Gamma=\Delta\ent$ is derivable.
\item Similarly, morphisms $[\Gamma]\to[\Delta]$ from the equivalence class of
$\Gamma$ to the equivalence class of $\Delta$ are substitutions
$\Gamma\ent\sigma:\Delta$ modulo derivable equality. (The closure conditions on
the set of derivable judgments ensure independence of representatives, e.g.\
that $\Gamma'\ent\sigma:\Delta'$ whenever $\Gamma\ent\sigma:\Delta$ and
$\Gamma=\Gamma'\ent$ and $\Delta=\Delta'\ent$.)
\item Composition is given by substitution of representatives, and identities
are given by lists of variables:
\begin{itemize}
\item $[\Delta\ent\tau:\Theta]\circ[\Gamma\ent\sigma:\Delta]
=[\Gamma\sent\tau[\sigma]:\Theta]$
\item $\id_\Gamma=(\Gamma\ent(\vec x):\Gamma)$ where $\vec x$ is the list of
variables declared in $\Gamma$.
\end{itemize}
\end{itemize}
\end{definition}
The syntactic category $\mcc[\TT]$ of a GAT $\TT$ has the structure of a
\emph{contextual category}:
\begin{definition}\label{def:cxtcat}
A \emph{contextual category} consists of
\begin{enumerate}[label=(\arabic*)]
\item a small category $\mcc$ with a {grading} function $\deg:\mcc_0\to\N$ on
its objects, and 
\item a presheaf $\Ty:\mcc\op\to\Set$, together with
\begin{itemize}
\item an object $\Gamma.A$ and an arrow $p_A:\Gamma.A\to\Gamma$ for each
$\Gamma\in\mcc$ and $A\in\Ty(\Gamma)$, and 
\item an arrow $\sigma.A:\Delta.A\sigma\to\Gamma.A$ for each
$\Gamma\in\mcc$, $A\in\Ty(\Gamma)$, and $\sigma:\Delta\to\Gamma$, 
\end{itemize}
\end{enumerate}
such that:
    \vspace{-2ex}
\begin{enumerate}
\item\label{def:cxtcat:pb}
The square 
$
\tikzsetnextfilename{HhkthvtZ}
\begin{mytikzcd}[row sep = small]
    |[label={[overlay, label distance=-2.5mm]-45:\lrcorner}]|
    \Delta.A\sigma
        \ar[r,"\sigma.A"]
        \ar[d,"p_{A\sigma}"']
\pnc   \Gamma.A
        \ar[d,"\,p_A"]
\\  \Delta
        \ar[r,"\sigma"]
\pnc   \Gamma
\end{mytikzcd}
$
is a pullback for all $A\in\Ty(\Gamma)$ and $\sigma:\Delta\to\Gamma$.
\item The mappings $(\Gamma,A)\mapsto \Gamma.A$ and $(\sigma,A)\mapsto\sigma.A$
constitute a functor $\El(\Ty)\to\mcc$.
\item We have $\deg(\Gamma.A)=\deg(\Gamma)+1$ for all $\Gamma\in\mcc$ and
$A\in\Ty(\Gamma)$.
\item There is a unique object $\ecxt$ of degree $0$, and $\ecxt$
is terminal.
\item For all $\Gamma$ with $\deg(\Gamma)>0$ there is a unique
$(\Gamma_0,A)\in\El(\Ty)$ with $\Gamma=\Gamma_0.A$.
\end{enumerate}
\end{definition}
In the case of the syntactic category $\mcc[\TT]$ of a GAT $\TT$, the grading
assigns to each context its length, and $\Ty(\Gamma)$ is the set of `types in
context $\Gamma$', i.e.\ equivalence classes of type expressions $A$ such that
$\Gamma\ent A$ is derivable, modulo the equivalence relation of derivable
equality. The presheaf action is given by substitution. Given a type
$A\in\Ty(\Gamma)$, the extended context $\Gamma.A$ is given by $\Gamma,y{:}A$
obtained via the context formation rule
in~\eqref{eq:context-substitution-rules}, and $p_A$ is the substitution 
\begin{equation*}
\Gamma,y{:}A\;\ent\;(\vec x)\;:\;\Gamma
\end{equation*}
where $\vec x$ is the list of variables declared in $\Gamma$. For
$\sigma:\Gamma\to\Delta$ and $A\in\Ty(\Delta)$, the substitution $\sigma.A$ is
given by 
\begin{equation*}
\Gamma,y{:}A[\sigma]\sent (\sigma,x)\;:\;\Delta,y{:}A\,.
\end{equation*}
Then the fact that the square in Definition \ref{def:cxtcat}\ref{def:cxtcat:pb}
is a pullback follows from the substitution formation rule in
\eqref{eq:context-substitution-rules} together with the equality rules for
substitutions that can be found in the cited references.

The following describes the relationship between contextual
categories and clans.
\begin{proposition}\label{prop:clan-from-cxcat}
Every contextual category $\mcc$ admits a clan structure where the display maps
are the composites $p_{A_1}\circ \dots\circ p_{A_n}\circ i$ of a finite sequence
of projections and an isomorphism. \qed
\end{proposition}

\section{The fat small object argument for clans}\label{se:fsoa}

\subsection{Colimit decomposition formula and pushouts of sieves}

This subsection discusses two results that are needed in the proof of the fat
small object argument. 

\begin{theorem}[Colimit decomposition formula (CDF)]\label{thm:cdm}Let $\CC:\JJ\to\Cat$ be a small diagram in the $1$-category of small categories,
and let $D:\colim(\CC)\to\mfx$ be a diagram in a category $\mfx$ such that
\begin{enumerate}
\item for all $j\in\JJ$, the colimit $\colim_{c\in\CC_j}
D_{\sigma_jc}=\colim(\CC_j\xrightarrow{\sigma_j}\colim(\CC)\stackrel{D}\to\mfx)$
exists, and 
\item the iterated colimit $\colim_{j\in\JJ}\colim_{c\in\CC_j} D_{\sigma_jc}$
exists.
\end{enumerate}
Then $\colim_{j\in\JJ}\colim_{c\in\CC_j} D_{\sigma_jc}$ is a colimit of $D$.
\end{theorem}
\begin{proof}
Peschke and Tholen~\cite{peschke2020diagrams} give three proofs of this under
the additional assumption that $\mfx$ is cocomplete. The third proof (\Section
5.3, `via Fubini') easily generalizes to the situation where only the necessary
colimits are assumed to exist. We sketch a simplified argument here. Let
$\gcons\CC$ be the {covariant Grothendieck construction} of $\CC$, whose
projection $\gcons\CC\to\JJ$ is a split opfibration. Then $\colim(\CC)$ is the
`joint coidentifier' of the splitting, \ie there is a functor
$E:\gcons\CC\to\colim(\CC)$
such that for every category $\mfx$, the precomposition functor
\begin{equation*}
(-\circ E):[\colim(\CC),\mfx]\to [\gcons \CC,\mfx]
\end{equation*}
restricts to an {isomorphism} between the functor category $[\colim(\CC),\mfx]$
and the full subcategory of $[\gcons \CC,\mfx]$ on functors which send the
arrows of the splitting to identities. In particular, $(-\circ E)$ is fully
faithful and thus it induces an isomorphism 
\begin{equation*}
(\colim(\CC))(D,\Delta -)\congto(\gcons\CC)(D\circ E,\Delta -):\mfx\to\Set
\end{equation*}
of co-presheaves of cocones for every diagram $D:\colim(\CC)\to\mfx$. In other
words, $E$ is \emph{final}, which is the crucial point of the argument, and for
which Peschke and Tholen give a lengthier proof
in~\cite[Theorem~5.8]{peschke2020diagrams}.

Finality of $E$ implies that $D$ has a colimit if and only if $D\circ E$ has a
colimit, and the existence of the latter follows if successive left Kan
extensions along the composite $\gcons \CC\to \JJ \to 1$ exist. The first of
these can be computed as fiberwise colimit since $\gcons\CC\to\JJ$ is a split
cofibration~\cite[Theorem~4.6]{peschke2020diagrams}, which yields the inner term
in the double colimit in the proposition.
\end{proof}

In the following we use the CDF specifically  for {pushouts of sieve inclusions}
of posets. Recall that a \emph{sieve} (\aka \emph{downset} or \emph{lower set})
in a poset $P$ is a subset $U\subs P$ satisfying 
\begin{equation*}
x\in U \wedge y\leq x \simplies y\in U
\end{equation*}
for all $x,y\in P$. A monotone map $f:P\to Q$ is called a
\emph{sieve inclusion} if it is order-reflecting and its image $\image(f)=f[P]$
is a sieve in $Q$. The proof of the following lemma is straightforward, but we
state it explicitly since it will play a central role.
\begin{lemma}\label{lem:sieve-pushout} 
\begin{enumerate}
\item\label{lem:sieve-pushout:plus}
If $f:P\to Q$ and $g:P\to R$ are sieve
inclusions of posets, a pushout of $f$ and $g$ in  the $1$-category $\Cat$ of
small categories is given by 
\tikzsetnextfilename{OEIjjqAC}
\begin{equation*}\begin{mytikzcd}
{}  P
        \ar[r,"g"]
        \ar[d,"f"']
\pnc  R
        \ar[d,"\sigma_2"]
\\  Q
        \ar[r,"\sigma_1"]
\pnc  |[label={[overlay, label distance=-2.5mm]135:\ulcorner}]| Q+_{P}R
\end{mytikzcd}\end{equation*}
where $Q+_{P}R$ is the set-theoretic pushout, ordered by 
\begin{alignat*}{2}
\sigma_1(x)&\leq \sigma_1(y) &\qquad&\text{iff }\; x\leq y\\
\sigma_1(x)&\leq \sigma_2(y) &&\text{iff }\; \ex z\qdot x=f(z)\wedge g(z)\leq y \\
\sigma_2(x)&\leq \sigma_2(y) &&\text{iff }\; x\leq y\\
\sigma_2(x)&\leq \sigma_1(y) &&\text{iff }\; \ex z\qdot x=g(z)\wedge f(z)\leq y .
\end{alignat*}
In particular, the maps $\sigma_1$ and $\sigma_2$ are also sieve inclusions.
\item\label{lem:sieve-pushout:uv}
If $U$ and $V$ are sieves in a poset $P$ then the square
\tikzsetnextfilename{FVjMgoGq}
\begin{equation*}\begin{mytikzcd}[sep = small]
U\cap V
        \ar[r]
        \ar[d]
\pnc	V
        \ar[d]
\\	U
        \ar[r]
\pnc  |[label={[overlay, label distance=-2.5mm]135:\ulcorner}]| U\cup V
\end{mytikzcd}
\end{equation*}
is a pushout in $\Cat$, where the sieves are equipped with the induced ordering.
\end{enumerate}
 \qed
\end{lemma}

\subsection{The fat small object argument}

Throughout this subsection let $\coclanc$ be a coclan. 

\medskip

We start by establishing some notation. Given a poset $P$ and an element $x\in
P$, we write $\idl{P}{x} = \setof{y\in P}{y\leq x}$ for the principal sieve
generated by $x$, and $\sidl{P}{x} = \setof{y\in P}{y< x}$ for its subset on
elements that are strictly smaller than $x$. If $x$ is a maximal element of $P$,
we write $P\setmin x$ for the sub-poset obtained by removing $x$. Given a
diagram $D:P\to\coclanc$, we write $\idl D x$, $\sidl D x$, and $D\setmin x$ for
the restrictions of $D$ to $\idl P x$, $\sidl P x$, and $P\setmin x$,
respectively. More generally we write $D_U$ for the restriction of $D$ to
arbitrary sieves $U\subs P$.

Note that we have $\idl P x = \sidl P x \star 1$, where $\star$ is the
\emph{join} or \emph{ordinal sum}, thus diagrams $D : \idl P x\to\coclanc$ are
in correspondence with cocones on $\sidl D x$ with vertex $D_x$, and with arrows
$\colim(\sidl D x)\to D_x$ whenever the colimit exists.

\begin{definition}\label{def:fcompl} A \emph{finite $\coclanc$-complex} is a
pair $\PD$ of a finite poset $P$ and a diagram $D:P\to\coclanc$, such that:
\begin{enumerate}
\item\label{def:fcompl:codisp}$\colim(\sidl D x)$ exists for all $x\in P$, and the induced $\alpha_x :
\colim(D_{<x}) \to D_x$ is \codisplay.
\item\label{def:fcompl:eq}For $x,y\in P$ we have $x=y$ whenever $P_{<x}=P_{<y}$, $D_x=D_y$, and $\alpha_x=\alpha_y$. 
\end{enumerate}
An \emph{inclusion of finite $\coclanc$-complexes} $f:\PD\to \QE$ is a sieve
inclusion $f:P\to Q$ such that $D=E\circ f$. We write $\FC(\coclanc)$ for the
category of finite $\coclanc$-complexes and inclusions.
\end{definition}
\begin{remark}
We view a \fcc as a construction of an object by a finite (though not
necessarily linearly ordered) number of `cell attachments', represented by the
\codisplay maps $\alpha_x:\colim(D_{<x})\codito D_x$. Condition
\ref{def:fcompl:eq} should be read as saying that `every cell can only be
attached once at the same stage'. This is needed in Lemma~\ref{lem:slat} to show
that $\FCC$ is a preorder.
\end{remark}
\begin{lemma}\label{lem:fcomp-col-ex}
\begin{enumerate}
\item The colimit $\colim(D)$ exists for every \fcc $\PD$.
\item The induced functor 
\begin{equation}\label{eq:colim-fun}
\Colim : \FCC\to\coclanc
\end{equation}
sends inclusions of \fccs
to \codisplay maps.
\end{enumerate}
\end{lemma}
\begin{proof}
The first claim is shown by induction on $\abs{P}$. For empty $P$ the statement
is true since coclans have initial objects. For $\abs{P}=n+1$ assume that $x\in
P$ is a maximal element. Then the square 
\tikzsetnextfilename{YRQlDhMK}
\begin{equation*}
\begin{mytikzcd}[row sep = small]
    \sidl P x
        \ar[d]
        \ar[r]
\pnc	
    P\setmin x
        \ar[d]
\\
    \idl P x
        \ar[r]
\pnc
    |[label={[overlay, label distance=-3mm]135:\ulcorner}]|P
\end{mytikzcd}
\end{equation*}
is a pushout in $\Cat$ by Lemma~\ref{lem:sieve-pushout}, which by the {colimit
decomposition formula}~\ref{thm:cdm} means that the pushout of the span
\begin{equation}\label{eq:cdf-square}
\tikzsetnextfilename{AqJPLRGI}
\begin{mytikzcd}
    \colim(\sidl{D}{x})
        \ar[d,{Triangle[open,reversed]->}]
        \ar[r]
\pnc	\colim(D\setmin x)
        \ar[d,dashed,{Triangle[open,reversed]->}]
\\	D_x
        \ar[r,dashed]
\pnc   |[label={[overlay, label distance=-3mm]135:\ulcorner}]|\colim(D)
\end{mytikzcd}\end{equation}
---which exists since the left arrow is a \codisplay map by
\ref{def:fcompl}\ref{def:fcompl:codisp}---is a colimit of $D$ in $\coclanc$.

\medskip

For the second claim let $f:(E,Q)\to(D,P)$ be an inclusion of \fccs. Since every
inclusion of \fccs can be decomposed into `atomic' inclusions with
$\abs{P\setmin f[Q]}=1$, we may assume without loss of generality that
$Q=P\setmin x$ for some maximal $x\in P$. Then the image of $f$ under $\Colim$
is the right dashed arrow in \eqref{eq:cdf-square}, which is \codisplay since
\codisplay maps are stable under pushout.
\end{proof}
\begin{remark}
Lemma~\ref{lem:fcomp-col-ex} implies that the assumption `$\colim(\sidl D x)$
exists' in Definition~\ref{def:fcompl}\ref{def:fcompl:codisp} is redundant,
since the colimits in question are colimits of finite subcomplexes.
\end{remark}

\begin{lemma}\label{lem:slat}The category $\FCC$ is an essentially small preorder with finite joins.
\end{lemma}
\begin{proof} $\FC(\coclanc)$ is essentially small as a collection of finite
diagrams in a small category. To see that it is a preorder let
$f,g:\PD\to\QE$ be inclusions of \fccs. We show that $f(x) = g(x)$ by
well-founded induction on $x\in P$. Let $x\in P$ and assume that
$f(y)=g(y)$ for all $y<x$.
Then since $f$ and $g$ are sieve
inclusions we have $\sidl Q {f(x)} = \sidl Q {g(x)}$ and since $E\,f=D=E\,g$ we
have the equalities 
\begin{equation*}
\lrfami{E_y\to E_{f(x)}}{y<f(x)}
=\lrfami{D_y\to D_{x}}{y<x}
=\lrfami{E_y\to E_{g(x)}}{y<g(x)}
\end{equation*}
of cocones, whence $f(x)=g(x)$ by
Definition~\ref{def:fcompl}\ref{def:fcompl:eq}.

\medskip

It remains to show that $\FC(\coclanc)$ has finite suprema. The empty complex is
clearly initial. We show that a supremum of $\PD$ and $\QE$ exists by induction
on $\abs{P}$. The empty case is trivial, so assume that $P$ is inhabited and let
$x$ be a maximal element. Let $(R,F)$ be a supremum of $(P\setmin x, D\setmin
x)$ and $\QE$, with inclusion maps $f:(P\setmin x, D\setmin x)\to (R,F)$ and
$g:\QE\too(R,F)$. If there exists a $y\in R$ such that $\sidl R y = f[\sidl P
x]$ and $\lrfami{D_z\to D_x}{z<x} = \lrfami{R_{f(z)}\to R_y}{z<x}$ then `the
cell-attachment corresponding to $x$ is already contained in $(R,F)$', \ie $f$
extends to an inclusion $f':\PD\to(R,F)$ of finite complexes with $f'(x)=y$,
whence $(R,F)$ is a supremum of $\PD$ and $\QE$.

If no such $y$ exists then a supremum of $\PD$ and $(R,F)$ is given by 
$(P+_{P\setmin x}R,[D,F])$, as in the pushout diagram
\tikzsetnextfilename{jFmWHSSM}
\begin{equation*}\begin{mytikzcd}
    P\setmin x
        \ar[r,"f"']
        \ar[d,"", hook]
\pnc	R
        \ar[d,""]
        \ar[rdd, bend left=20,"F" description]
\\	P
        \ar[r,""]
        \ar[rrd, bend right=15,"D" description]
\pnc	|[label={[overlay, label distance=-3mm]135:\ulcorner}]| P+_{P\setmin x}R 
\ar[rd,"{[D,F]}" description, dashed]
\\[-20pt]\pnc\pnc[+20pt]
    \coclanc
\end{mytikzcd}\end{equation*}
constructed as in Lemma~\ref{lem:sieve-pushout}.
\end{proof}

\begin{theorem}\label{thm:ext-full-fact}
The object
$\textstyle C = \colim_{\PD\in\FCC}H(\colim(D))$
is a \z-extension in $\copalg$ and $C\to 1$ is full.
\end{theorem}
\begin{proof}
To see that $C\to 1$ is full, let $e:I\codito J$ be \codisplay in $\coclanc$ and
let $f:H(I)\to C$. Since $\FC(\coclanc)$ is filtered and $H(I)$ is compact, $f$
factors through a colimit inclusion as \begin{equation*}
f\;=\;\bigl(H(I)\xrightarrow{H(g)}H(\colim(D))\xrightarrow{\sigma_\PD}C\bigr)
\end{equation*}
for some finite complex $\PD$. We form the pushout
\tikzsetnextfilename{MLoZOdts}
\begin{equation*}\begin{mytikzcd}
    I
        \ar[r,"g"]
        \ar[d,"e"',{Triangle[open,reversed]->}]
\pnc	\colim(D)
        \ar[d,"k",{Triangle[open,reversed]->}]
\\	J
        \ar[r,""]
\pnc	|[label={[overlay, label distance=-3mm]135:\ulcorner}]| K
\end{mytikzcd}\end{equation*}
and extend the finite complex $\PD$ to $(P\star 1,D\star k)$ where $P\star 1$ is
the {join} of $P$ and $1$, and $D\star k:P\star 1 \to \coclanc$ is the 
diagram extending $D$ with the cell-attachment
$k:\colim(D)\codito K$. Then $K=\colim(D\star k)$ and $k$ is the image of the
inclusion $\PD\incl(P\star 1,D\star k)$ of finite complexes under the colimit
functor~\eqref{eq:colim-fun}, thus we obtain an extension of $f$ along $H(e)$ as
in the following diagram. 
\tikzsetnextfilename{phkHjiTY}
\begin{equation*}\begin{mytikzcd}
    H(I)
        \ar[r,"H(g)"']
        \ar[d,"H(e)"',{Triangle[open,reversed]->}]
        \ar[rr, bend left = 17, "f" description]
\pnc	H(\colim(D))
        \ar[d,{Triangle[open,reversed]->},"H(k)"]
        \ar[r,"\sigma_\PD"']
\pnc   C
\\  H(J)
        \ar[r,""]
\pnc   |[label={[overlay, label distance=-3mm]135:\ulcorner}]| H(K)
        \ar[ur,"\sigma_{(P\star 1,D\star k)}"', bend right = 17]
\end{mytikzcd}\end{equation*}

To see that $C$ is a \z-extension, consider a full map $f:Y\fullmap X$ in
$\copalg$ and an arrow $h:C\to X$. To show that $h$ lifts along $f$ we construct
a lift of the cocone
\begin{equation*}
\left(H(\colim(D))\xrightarrow{\sigma_\PD}C\xrightarrow{h}X\right)_{\PD\in\FCC}
\end{equation*}
by induction over the preorder $\FC(\coclanc)$ which is well-founded since every
\fcc has only finitely many subcomplexes. Given a finite complex $(D,P)$ it is
sufficient to exhibit a lift $\kappa_\PD:H(\colim (D))\to Y$ satisfying 
\begin{align}
\quad&f\circ\kappa_\PD=h\circ\sigma_\PD&&\text{and}\label{eq:fkappahsigma}\\ 
\quad&\kappa_\PD\circ H(\colim j)= \kappa_\QE 
    &&\text{for all subcomplexes }j:\QE\to\PD\label{eq:kappahcolimkappa},
\end{align}
where we may assume that the $\kappa_\QE$ satisfy the analogous equations by
induction hypothesis. 
We distinguish two cases: 

1. If $P$ has a greatest element $x$ then we can take
$\kappa_\PD$ to be a lift in the square 
\tikzsetnextfilename{xlgeLlVc}
\begin{equation*}\begin{mytikzcd}
    H(\colim (\sidl D x))
        \ar[rr,"\kappa_{(\sidl P x, \sidl D x)}"]
        \ar[d,""]
\pnc\pnc	Y
        \ar[d,"f"]
\\	H(D_x)
        \ar[r,"\sigma_\PD"']
        \ar[rru,dashed,"\kappa_\PD" description]
\pnc   C   \ar[r,"h"']
\pnc	X
\end{mytikzcd}\end{equation*}
whose left side is an extension by Lemma~\ref{lem:fcomp-col-ex} and whose right
side is full by assumption. Then \eqref{eq:fkappahsigma} holds by construction,
and \eqref{eq:kappahcolimkappa} holds for all subcomplexes since it holds for
the largest strict subcomplex $(\sidl P x,\sidl D x)\to\PD$.

2. If $P$ \emph{doesn't} have a greatest element we can write $P=U\cup V$ as
union of two strict sub-sieves, whence we have pushouts 
\tikzsetnextfilename{sdOMglbF}
\begin{equation*}
\begin{mytikzcd}[sep = small]
U\cap V
        \ar[r]
        \ar[d]
\pnc   V
        \ar[d]
\\  U
        \ar[r]
\pnc   |[label={[overlay, label distance=-3mm]135:\ulcorner}]| P
\end{mytikzcd}
\qquad\qtext{and}\qquad
\tikzsetnextfilename{abHzgDyV}
\begin{mytikzcd}[sep = small]
\colim(D_{U\cap V})
        \ar[r]
        \ar[d]
\pnc   \colim(D_V)
        \ar[d]
\\  \colim(D_U)
        \ar[r]
\pnc   |[label={[overlay, label distance=-3mm]135:\ulcorner}]| \colim(D)
\end{mytikzcd}
\end{equation*}
by Lemma~\ref{lem:sieve-pushout} and the CDF. This means that condition
\eqref{eq:kappahcolimkappa} forces us to define $\kappa_\PD$ to be the unique
arrow fitting into
\begin{equation}\label{eq:kpd}
\tikzsetnextfilename{CAjSuDlV}
\begin{mytikzcd}
    H(\colim(D_{U\cap V}))
        \ar[r,"\phi^{U\cap V}_V\,"]
        \ar[d,"\phi^{U\cap V}_U"']
\pnc   H(\colim(D_{V}))
        \ar[d,"\phi^V_P"]
        \ar[rdd, bend left=20,"\kappa_{(V,D_V)}" description]
\\  H(\colim(D_{U}))
        \ar[r,"\phi^U_P"]
        \ar[rrd, bend right=10,"\kappa_{(U,D_U)}" description]
\pnc   |[label={[overlay, label distance=-3mm]135:\ulcorner}]|H(\colim(D))
        \ar[rd,"{\kappa_\PD}" description,dashed]
\\[-20pt]\pnc\pnc[+20pt]
    Y
\end{mytikzcd},\end{equation}
where for the remainder of the proof we write $\phi^X_W:H(\colim(D_X))\to
H(\colim(D_W))$ for the canonical arrows induced by successive sieve inclusions
$X\subs W\subs P$. Using the fact that the $\phi^U_P$ and $\phi^V_P$ are jointly
epic it is easy to see that the $\kappa_\PD$ defined in this way satisfies
condition \eqref{eq:fkappahsigma}, and it remains to show that
\eqref{eq:kappahcolimkappa} is satisfied for arbitrary sieves $W\subs P$, \ie
$\kappa_\PD\circ \phi^W_P = \kappa_{(W,D_W)}:H(\colim(D_W))\to Y$. Since 
\tikzsetnextfilename{YuLeQuei}
\begin{equation*}\begin{mytikzcd}[column sep = large]
    H(\colim(D_{U\cap V\cap W}))
        \ar[r,"\phi^{U\cap V\cap W}_{V\cap W}\,"]
        \ar[d,"\phi^{U\cap V\cap W}_{U\cap W}"']
\pnc	H(\colim(D_{V\cap W}))
        \ar[d,"\phi^{V\cap W}_W"]
\\	H(\colim(D_{U\cap W}))
        \ar[r,"\phi^{U\cap W}_W"]
\pnc	|[label={[overlay, label distance=-3mm]135:\ulcorner}]| H(\colim(D_W))
\end{mytikzcd}\end{equation*}
is a pushout it is enough to verify this equation after precomposing with
$\phi^{U\cap W}_W$ and $\phi^{V\cap W}_W$. We have
\begin{align*}
\kappa_\PD\circ\phi^W_P\circ\phi^{U\cap W}_W
&=\kappa_\PD\circ\phi^U_P\circ\phi^{U\cap W}_U &&\text{by functoriality}\\
&=\kappa_{(U,D_U)}\circ\phi^{U\cap W}_U &&\text{by \eqref{eq:kpd}}\\
&= \kappa_{(U\cap W,D_{U\cap W})} && \text{by \eqref{eq:kappahcolimkappa}}\\
&=\kappa_{(W,D_W)}\circ\phi^{U\cap W}_W&& \text{by \eqref{eq:kappahcolimkappa}}
\end{align*}
and the case with $\phi^{V\cap W}_W$ is analogous.
\end{proof}

\begin{corollary}\label{cor:zext-flat} For any clan $\clant$, the \z-extensions
in $\tmod$ are flat.
\end{corollary}
\begin{proof}
Let $E\in\tmod$ be a \z-extension. By applying Theorem~\ref{thm:ext-full-fact}
in $\tmod/E$ (using Proposition~\ref{prop:slice-tmod}), we obtain a full map
$f:F\epi E$ where $F$ is a \z-extension and $f$ is a filtered colimit of arrows
$H(\Gamma)\to E$ in $\tmod/A$. Since $\tmod/A\to\tmod$ creates colimits this
means that $F$ is a filtered colimit of representable models in $\tmod$, and therefore
flat (Lemma~\ref{lem:flat-filtered}). Since $f$ is a full map into a
\z-extension it has a section, thus $E$ is a retract of $F$ and therefore flat
as well.
\end{proof}

\nocite{leena2021dependent}

\ifpublish
\bibliographystyle{asl}
\else
\bibliographystyle{alpha}
\fi

\end{document}